\documentclass[12pt]{amsart}  
\usepackage[margin=.85 in]{geometry}
\usepackage{tikz-cd}
\usepackage{amscd}
\usepackage{amssymb}
\usepackage{thmtools} 
\usepackage{thm-restate}
\usepackage{dsfont}
\usepackage{bbm}
\usepackage[pdfencoding=auto, psdextra]{hyperref}
\usepackage{mathrsfs}
\usepackage{scalerel}
\theoremstyle{plain}
\newtheorem{lem}{Lemma}[subsection]
\newtheorem{prop}[lem]{Proposition}
\newtheorem{thm}[lem]{Theorem}
\newtheorem{cor}[lem]{Corollary}

\theoremstyle{definition}
\newtheorem{const}{Construction}[subsection]
\newtheorem{defi}[const]{Definition}

\theoremstyle{remark}
\newtheorem*{remark}{Remark}
\newtheorem{example}{Example}[subsection]
\newtheorem{nota}{Notation}[section]\newcommand{\trop}{\textnormal{trop}}

\numberwithin{equation}{section}

\newcommand{\legval}{\operatorname{legval}}
\newcommand{\HS}{\operatorname{HS}}
\newcommand{\VS}{\operatorname{VS}}

\newcommand{\bbG}{\mathbb{G}}

\newcommand{\bbP}{\mathbb{P}}
\newcommand{\bbQ}{\mathbb{Q}}
\newcommand{\bbR}{\mathbb{R}}

\newcommand{\bbZ}{\mathbb{Z}}

\newcommand{\cM}{\mathcal{ M}}

\newcommand{\cX}{\mathcal{ X}}
\newcommand{\cY}{\mathcal{ Y}}

\newcommand{\op}{\textnormal{op}}

\newcommand{\Hom}{\operatorname{Hom}}

\newcommand{\Top}{\operatorname{Top}}

\newcommand{\Lin}{\operatorname{Lin}}

\newcommand{\rk}{\operatorname{rk}}

\newcommand{\Div}{\operatorname{Div}}
\newcommand{\val}{\operatorname{val}}

\newcommand{\Rat}{\operatorname{Rat}}

\newcommand{\Mtrop}{\operatorname{M}^{\trop}}
\renewcommand{\Im}{\operatorname{Im}}
\newcommand{\rank}{\operatorname{rk}}
\newcommand{\ord}{\operatorname{ord}}

\renewcommand{\div}{\operatorname{div}}

\newcommand{\POIC}{\operatorname{POIC}}

\newcommand{\ST}{\operatorname{ST}}
\newcommand{\st}{\operatorname{st}}
\newcommand{\est}{\operatorname{est}}

\newcommand{\bbAC}{\mathbb{A}\mathbb{C}}
\newcommand{\bbEAC}{\mathbb{E}\mathbb{A}\mathbb{C}}
\newcommand{\src}{\textnormal{src}}
\newcommand{\trgt}{\textnormal{trgt}}

\newcommand{\ft}{\operatorname{ft}}

\newcommand{\tint}{\textnormal{int}}
\newcommand{\pspcs}{\textnormal{s}}
\newcommand{\pcmplxs}{\textnormal{c}}
\newcommand{\AC}{\mathrm{AC}}
\newcommand{\EST}{\mathrm{EST}}
\newcommand{\Eff}{\mathrm{Eff}}

\newcommand{\Subd}{\operatorname{Subd}}
\newcommand{\mt}{\textnormal{mult}}
\title{Enumeration of Weierstrass points of metric graphs.}
\author{Diego A. Robayo Bargans}
\address{FB Mathematik, RPTU Kaiserslautern, 67663 Kaiserslautern, Germany}
\email{robayo@mathematik.uni-kl.de}

\begin{document}

\begin{abstract}
    A classical result states that on a smooth algebraic curve of genus $g$ the number of Weierstrass points, counted with multiplicity, is $g^3-g$. In this paper, we introduce the notion of geometric Weierstrass points of metric graphs and show that a generic metric graph of genus $g$ has $g^3-g$ geometric Weierstrass points counted with multiplicity. Our methods also provide a new proof of the existence of Weierstrass points on metric graphs of genus bigger than or equal to $2$.
\end{abstract}

\maketitle
\setcounter{tocdepth}{2}
\section{Introduction}

The divisor theory of discrete and metric graphs has been developed in close parallel to its classical algebrogeometric analog, so that several results and constructions mimic their classical counterpart (see \cite{BakerSpecialization}, \cite{BakerNorine}, \cite{GathmannKerberARRTITG} \cite{MikhalkinZharkovTCTJATF}). Following this mirroring spirit the corresponding notion of Weierstrass points for discrete and metric graphs was introduced by Baker in \cite{BakerSpecialization}. Namely, for a metric graph $\Gamma$ of genus $g$, a point $p\in \Gamma$ is a \emph{Weierstrass point} if the divisor $g\cdot [p]$ has Baker-Norine rank greater than or equal to $1$. The classical enumerative result states that on a smooth algebraic curve of genus $g$ the number of Weierstrass points, counted with multiplicity, is $g^3-g$. Such result does not hold in this setting, where, in contrast with the classical situation, there are usually infinitely many Weierstrass points on a given metric graph. In this paper, we introduce a special kind of Weierstrass points on metric graphs, geometric Weierstrass points, and make use of the technical framework from \cite{DARB1} and \cite{DARB2} to obtain a tropical intersection theoretic proof of the existence of Weierstrass points on metric graphs. Additionally, we also the obtain the corresponding enumerative statement with these geometric Weierstrass points in a similar vein to that of Theorem 4.3.3 of \cite{DARB2}. To be precise, we show that: \emph{A generic genus $g$ metric graph has $g^3-g$ geometric Weierstrass points counted with multiplicity}. 

In what follows, by a tropical curve is meant a metric space given by a discrete graph (possibly with legs) with a metric, and by a metric graph is meant a compact tropical curve. We consider discrete admissible covers of marked trees in order to obtain divisors of positive rank on metric graphs. More specifically, a discrete admissible cover of a marked tree gives rise to a tropical admissible cover of a metric tree, and the pullback through a tropical admissible cover of any degree-$1$ effective divisor on the metric tree produces a divisor with rank at least $1$ on the source. This procedure is related to Weierstrass points when we consider discrete admissible covers where the degree coincides with the genus of the source and there is a leg of the source whose weight equals the degree. In this set-up, the vertex incident to the latter leg will determine a Weierstrass point of the metric graph arising from the source, and we refer to Weierstrass points that arise in this way as geometric Weierstrass points. To further study the metric graphs that have geometric Weierstrass points, and how many there could be, we are directly prompted towards the framework of \cite{DARB2} and we use a special case of Theorem 4.1.1 of loc. cit. to deduce our enumerative statement.

In more detail, we will consider the poic-spaces $\Mtrop_g$ and $\AC_{g,0,\vec{\mu}}$ that parameterize, respectively, all genus-$g$ metric graphs and all degree-$g$ discrete admissible covers of $3g$-marked trees with ramification profiles above the legs given by $\vec{\mu}=(\mu_1,\dots,\mu_{3g})$ where $g\geq1$ and  
\begin{align*}
    &\mu_1= (g),&&\mu_i=(2,1^{g-2}),\textnormal{ for }2\leq i\leq 3g.
\end{align*}
We are interested in keeping the source of these discrete admissible covers and forgetting all the markings of the source curve to produce a genus-$g$ metric graph with a geometric Weierstrass point. For this reason, we consider the spanning tree fibration $\st_g$ over $\Mtrop_g$ and the extended spanning tree fibration $\est^\varnothing_{g,0,\vec{\mu}}$ over $\AC_{g,0,\vec{\mu}}$ (here $\varnothing$ indicates that we will eventually forget all the marked legs of the source). These poic-fibrations allow us to study tropical cycles on $\Mtrop_g$ and $\AC_{g,0,\vec{\mu}}$, and by means of these we are able to deduce our enumerative statement. The standard weight $\varpi^\varnothing_{g,0,\vec{\mu}}$ (Definition \ref{eq: standard weight}) defines an $\est^\varnothing_{g,0,\vec{\mu}}$-equivariant tropical cycle (Theorem 4.1.1 of loc. cit.), and the poic-fibration $\est^\varnothing_{g,0,\vec{\mu}}$ comes with the source morphism $\mathfrak{src}\colon \est^\varnothing_{g,0,\vec{\mu}}\to \st_g$ through which we pushforward the cycle defined by $\varpi^\varnothing_{g,0,\vec{\mu}}$. We show that this pushforward is a non-trivial multiple of the fundamental cycle of $\st_g$ and compute the corresponding multiplicity. Afterward, we define an equivalence relation through permutations of the markings of the discrete admissible covers whose equivalence classes correspond to the different Weierstrass points on a given metric graph. Through the previous results we introduce a multiplicity on this set of equivalence classes, and show that on a generic genus-$g$ graph counting this equivalence classes with multiplicity gives $(g+1)\cdot g\cdot(g-1)$.

The structure of the paper is as follows: In Section 2 we gather some preliminaries from \cite{DARB1} and \cite{DARB2}, and the reader is pointed toward these references for examples, details, and proofs. In Section 3 we compile several constructions and results from the divisor theory of metric graphs and their relation with tropical admissible covers. In Section 4 we discuss the notion of Weierstrass points of metric graphs, introduce that of geometric Weierstrass points, and apply the machinery of \cite{DARB1} and \cite{DARB2} to produce our previously mentioned enumerative result.

\subsection*{Acknowledgements:}
I would like to thank A. Gathmann for helpful and inspiring discussions, proofreading, and commentary. I also thank A. Vargas for helpful conversations and discussions.

\section{Preliminaries}
We recall the relevant notions, constructions, and results from \cite{DARB1} and \cite{DARB2}. In the spirit of brevity, the reader is pointed to these references for examples, proofs, and much more detail.

\subsection{Discrete graphs and discrete admissible covers} We begin by fixing notation and terminology surrounding discrete graphs and discrete admissible covers.

\begin{defi}
    A \emph{discrete graph} $G$ consists of the data $(F(G),r_G,\iota_G)$ where:
    \begin{itemize}
        \item $F(G)$ is a finite set.
        \item $r_G:F(G)\to F(G)$ is a map of sets (the root map).
        \item $\iota_G:F(G)\to F(G)$ is an involution, such that $\iota_G\circ r_G=r_G$.
    \end{itemize}
The \emph{set of vertices} of $G$ is the subset $V(G)$ of $F(G)$ given by $\Im(r_G)$, and its complement is denoted by $H(G)$.
\end{defi}

For a discrete graph $G$ the involution $\iota_G$ of $F(G)$ restricts to an involution of $H(G)$, and hence partitions $H(G)$ into orbits of size $1$ and $2$. The following definitions will be employed:
\begin{itemize}
    \item A \emph{leg} of $G$ is a size $1$ orbit of $H(G)$. The set of legs will be denoted by $L(G)$, and the \emph{boundary of a leg} $\ell\in L(G)$ is the subset $\partial e :=r_G(\ell) \subset V(G)$.
    \item An \emph{edge} of $G$ is a size $2$ orbit of $H(G)$. The set of edges will be denoted by $E(G)$, and the \emph{boundary of an edge} $e\in E(G)$ is the subset $\partial e:= r_G(e)\subset V(G)$.
    \item A vertex is \emph{incident} to an edge or a leg if it belongs to its boundary.
    \item The \emph{valency} of $V\in V(G)$ is simply $\val(V) := \# r_G^{-1}(V)-1$. 
    \item The \emph{leg valency} of $V\in V(G)$ is $\legval(V):= \#r_G^{-1}(V)\cap L(G)$.
    \item An edge is called a \emph{leaf} if it has an incident vertex with leg valency $2$.
    \item A \emph{subgraph} $K$ of $G$ is a graph $(F(K),r_K,\iota_K)$ such that $F(K)\subset F(G)$ with $r_K=r_G|_{F(K)}$ and $\iota_K=\iota_G|_{F(K)}$.
\end{itemize}

\begin{defi}\label{defi: maps of graphs}
 Suppose $G_1$ and $G_2$ are discrete graphs. A \emph{map of graphs} $f\colon G_1\to G_2$ is a map of sets $f\colon F(G_1)\to F(G_2)$ that commutes with the root maps (i. e. $r_{G_2}\circ f = f\circ r_{G_1}$) and the involution maps (i. e. $\iota_{G_2}\circ f = f\circ \iota_{G_1}$). 
\end{defi}

A map of graphs $f\colon G_1\to G_2$ is called:
\begin{itemize}
    \item \emph{bijective}, if the underlying map of sets is bijective.
    \item \emph{non-contracting}, if it does not involve the contraction of edges or legs of $G_1$ to vertices (i. e. $f^{-1}(V(G_2))\subset V(G_1)$).
\end{itemize}

We follow the notation of \cite{DARB2} for the contraction of edges and subgraphs of a graph.
\begin{defi}\label{defi: contraction of legless subgraph}
    If $K$ be a legless subgraph of $G$, with $K=K_1\sqcup \dots \sqcup K_m$ its connected components, then the \emph{contraction of $K$ in $G$} is the graph $G/K$ obtained by the successive contractions of all the edges of $K_1, \dots ,K_m$, where the new vertex obtained from the contraction of all the edges of $K_i$ is denoted by $V_{K_i}$.
\end{defi}
The reader is kindly referred to Lemma 2.2.3 of \cite{DARB1} for the proof of the following lemma.
\begin{lem}\label{lem: map of graphs is contraction}
    Suppose $f\colon G_1\to G_2$ is a map of graphs such that:
    \begin{enumerate}
        \item For every $V\in V(G_2)$, the inverse image $f^{-1}(V)\subset F(G_1)$ is a legless tree.
        \item For every $e\in E(G_2)\cup L(G_2)$, the inverse image under $f$ consists of a single edge or leg.
    \end{enumerate}
    Then there is a natural bijective map $G_1/f^{-1}\left(V(G_2)\right)\to G_2$.
\end{lem}

\begin{defi}
    A \emph{degree assignment} on $G_1$ is a map $d\colon F(G_1)\to \bbZ_{\geq0}$ such that $d\circ \iota = d$.
\end{defi}

\begin{nota}
    A degree assignment on $G_1$ is equivalently defined as an integral valued function on $V(G_1)\cup E(G_1)\cup L(G_1)$ due to its behavior with the involution. We abuse this equivalence and alternate indiscriminately between these two equivalent definitions.
\end{nota}

\begin{defi}
     A \emph{discrete admissible cover} $\pi\colon G_1\to G_2$ consists of the following data:
     \begin{itemize}
         \item A surjective non-contracting map of graphs $\pi\colon G_1\to G_2$. 
         \item A degree assignment $d_\pi$ on $G_1$ satisfying: For any $V\in V(G_1)$ and $H^\prime\in r_{G_2}^{-1}\left(\pi(V)\right)\backslash\{\pi(V)\}$ 
     \begin{equation*}
     d_\pi(V) = \sum_{H\in \pi^{-1}(H^\prime)\cap r_{G_1}^{-1}(V)}d_\pi(H).
     \end{equation*}
        \item These are subject to the following condition (vanishing of the Riemann-Hurwitz number): For every $V\in V(G)$ we have that
        \begin{equation*}
            \val(V)-2 = d_\pi(V)\cdot\left(\val\pi(V)-2\right).
        \end{equation*}
     \end{itemize}
    If $\pi\colon G_1\to G_2$ is a discrete admissible cover, then the \emph{local degree at } $V\in V(G_1)$ is the integer $d_\pi(V)$.
\end{defi}
It is well known that the second condition on a discrete admissible cover $\pi\colon G\to H$ implies that the weighted sum of fibers of vertices of $H$ is a constant number (see for instance Lemma 3.1.1 of \cite{DARB2}). This gives rise to the next definition.
\begin{defi}
    The \emph{degree} of a discrete admissible cover $\pi\colon G\to H$ is the integer $\deg (\pi)$ given as $\deg \pi := \sum_{V\in\pi^{-1}(A)}d_\pi(V)$, for an arbitrary $A\in V(H)$.
\end{defi}

The degree of a discrete admissible cover is related to the genera of the source and target as follows (Proposition 3.1.2 of \cite{DARB2}).

\begin{prop}[Riemann-Hurwitz equality] Suppose $G_2$ is connected and $\pi\colon G_1\to G_2$ is a discrete admissible cover. The genera $g(G_1)$ and $g(G_2)$ are related to the degree of $\pi$ by:
\begin{equation}
\#L(G_1) + 2(g(G_1)-1) = \deg\pi \cdot ( \#L(G_2) + 2(g(G_2)-1)) + \sum_{V\in V(G_1)} r_\pi(V).\label{eq: Riemann-Hurwitz}
\end{equation}
\end{prop}

\subsection{Essentials of poic-spaces, -fibrations, and their tropical cycles} 
In this section we recollect the definitions concerning poic-spaces, -complexes, and -fibrations necessary to our endeavors. We avoid detail and favor heuristics, so that the interested reader is referred toward \cite{DARB1} and \cite{DARB2} for the technicalities of the framework. We refrain from formalities and technicalities to favor a quicker and more direct introduction to the objects of interest. 

We begin with the definition of poics. These are the building blocks behind poic-spaces, and the subsequent objects.

\begin{defi}\label{defi: poic}
A \emph{partially open integral polyhedral cone} (hereafter poic) is a pair $(\sigma, N)$ where 
\begin{itemize}
    \item $N$ is a free abelian group of finite rank, 
    \item $\sigma$ is a partially open rational polyhedral cone of $N_\bbR$ with $\dim(\sigma)=\dim(N_\bbR)$.
\end{itemize}
A \emph{face} of a partially open integral cone $(\sigma,N)$ is $(\tau,\Lin_N(\tau))$ where $\tau$ is a face of $\sigma$ in $N_\bbR$ and $\Lin_N(\tau) = N\cap \Lin(\tau)$. 
\end{defi}

We remark that if $(\sigma,N)$ is a poic, then $(\sigma^0,N)$ is also a poic where $\sigma^0$ denotes the relative interior. If the pair $(\sigma,N)$ is a poic, then we will denote $N$ by $N^\sigma$, and denote the pair simply by $\sigma$.

\begin{defi}
    If $\xi$ is an additional poic, then a \emph{morphism} $f\colon\sigma\to \xi$ consists of an integral linear map $f\colon  N^\sigma\to N^\xi$, such that the induced linear map $f_\bbR:N^\sigma_\bbR\to N^\xi_\bbR$ sends $\sigma$ to $\xi$. We say that a morphism $f\colon \sigma\to \xi$ is a \emph{face-embedding}, if the integral linear map is injective and $f_\bbR(\sigma)\leq \xi$. 
\end{defi}

The data of poics and morphisms thereof, together with natural composition of maps, defines a category, which we denote by $\POIC$. We now continue with poic-spaces. The idea behind a poic-space is that of topological spaces glued out of finitely many quotients of partially open polyhedral cones by actions of finite groups. 

\begin{defi}
    A \emph{poic-space} $\cX$ is a functor $\cX\colon C_\cX\to \POIC$, where $C_\cX$ is an essentially finite category with finite hom-sets, and $\cX$ is subject to the following conditions:
    \begin{itemize}
        \item The functor $\cX$ maps morphisms of $C_\cX$ to face-embeddings.
        \item If $x$ is an object of $C_\cX$ and $\tau$ is a face of $\cX(x)$, then there exists a unique up to isomorphism $w\to x$ in $C_\cX$ with $\cX(w\to x) \cong \tau\leq \cX(x)$.
    \end{itemize} 
    \end{defi}

We denote the set of isomorphism classes of $C_\cX$ by $[\cX]$, and the set of isomorphism classes of $k$-dimensional cones of $\cX$ by $[\cX](k)$. We say that a poic-space $\cX$ is \emph{pure of dimension $n$}, if every object $t$ of $\cX$ appears as the source of a morphism to an object of $\cX$ whose cone is $n$-dimensional.

\begin{defi}
A \emph{morphism of poic-spaces} $F\colon \cX\to\cY$ consists of the following data: 
    \begin{itemize}
        \item A functor $F\colon C_\cX\to C_\cY$.
        \item A natural transformation $\eta_F\colon \cX\implies \cY\circ F$.
    \end{itemize}
    These are subject to the following condition: for a cone $x$ of $\cX$, the cone $\eta_{F,x}\left(\cX(x)\right)$ does not lie in a proper face of $\cY(F(x))$. Composition of morphisms of poic-spaces is defined in the natural way.
\end{defi}

Of special interest are the poic-spaces that mimic the structure of partially open fans, namely where there are no finite groups acting on the cones. These are poic-complexes. More precisely:

\begin{defi}
A \emph{poic-complex} $\Phi$ is a poic-space, where the underlying category $C_\Phi$ is essentially finite and thin. A \emph{$k$-dimensional rational weight $\omega$ on $\Phi$} is simply a function $\omega\colon  [\Phi](k)\to \bbQ$.
\end{defi}

As usual, tropical cycles are described as weights (on a subdivision) satisfying a balancing condition. In our context, the balancing condition makes sense upon the introduction of a  morphism to a rational vector space, and this data gives rise to the notion of linear poic-complexes. 

\begin{defi}
A \emph{linear poic-complex} $\Phi_X$ consists of the data of:
\begin{itemize}
    \item A poic-complex $\Phi$. 
    \item A morphism of poic-complexes $X\colon \Phi\to \underline{(N_X)_\bbR}$, where $N_X$ is a free abelian group of finite rank and $\underline{(N_X)_\bbR}$ denotes the poic-space defined (tautologically) by the poic $((N_X)_\bbR,N)$.
\end{itemize}
\end{defi}

Suppose $\Phi_X$ is a linear poic-complex and let $k\leq0$ be an integer. The additional datum of the morphism $X\colon \Phi\to\underline{(N_X)_\bbR}$ gives rise to a notion of balancing, and therefore to \emph{$k$-dimensional rational Minkowski weights} of $\Phi_X$ which form a subgroup $M_k(\Phi_X)_\bbQ$ of $W_k(\Phi)_\bbQ$. We use rational Minkowski weights to define rational tropical cycles on linear poic-complexes. More specifically, if $S\colon\Phi^\prime\to\Phi$ is a \emph{subdivision}, then we obtain a linear poic-complex $\Phi^\prime_{X\circ S}$ and an injective linear map $S^*\colon M_k(\Phi_X)_\bbQ\to M_k(\Phi^\prime_{X\circ S})_\bbQ$. We consider the category $\Subd(\Phi)$ given by the \emph{subdivisions} of $\Phi$ (Definition 3.4.4 \cite{DARB1}) and observe that the above defines a functor with a representable colimit
\begin{equation*}
    M_k(\bullet)_\bbQ \colon \Subd(\Phi)^{\op}\to \bbQ\mathrm{-Vect},
\end{equation*} 
and we say that the \emph{group of rational tropical $k$-cycles} is the colimit of this functor (here $\bbQ\mathrm{-Vect}$ denotes the category of $\bbQ$-vector spaces) an denote it by $Z_k(\Phi_X)_\bbQ$.

\begin{defi}
    A linear poic-complex $\Phi_X$ pure of dimension $n$ (every cone maps to an $n$-dimensional cone) is called \emph{irreducible}, if $Z_n(\Phi_X)_\bbQ$ is $1$-dimensional.
\end{defi}

If $\Phi_X$ and $\Psi_Y$ are linear poic-complexes, then a \emph{morphism of linear poic-complexes} $\phi:\Phi_X\to\Psi_Y$ consists of the following data: 
\begin{itemize}
    \item A morphism of poic-complexes $\phi\colon\Phi\to\Psi$
    \item An integral linear map $\phi_{\tint}\colon N_X\to N_Y$ (this is equivalent to a morphism of poic-complexes $\underline{\left(N_X\right)_\bbR}\to \underline{\left(N_Y\right)_\bbR}$).
    \item These yield a commutative square of poic-complexes.
\end{itemize}
There are several technicalities concerning pushforward of tropical cycles, for which we refer the reader to Section 3.5 of \cite{DARB1}.

To describe cycles of a poic-space we make use of linear poic-fibrations. In this case, we speak of equivariant rational tropical cycles with respect to the poic-fibration. We briefly describe what these are, and the general framework around them.

\begin{defi}
     Let $\Phi$ be a poic-complex and $\cX$ a poic-space. A morphism of poic-spaces $\pi\colon \Phi\to \cX$ is called a \emph{poic-fibration}, if:
\begin{enumerate}
    \item The functor is essentially surjective.
    \item For any cone $p$ of $\Phi$ the map $\eta_{\pi,p}\colon  \Phi(p)\to \cX\left( \pi(p)\right)$ induces an isomorphism between their relative interiors.
    \item For any object $p$ of $\Phi$ and any morphism $f\colon \pi(p)\to x$ in $\cX$, there exists a morphism $h\colon p\to q$ in $\Phi$ and an isomorphism $g\colon  \pi(q)\to x$ such that $f=g\circ \pi(h)$, where $h$ is unique up to isomorphism under $p$, and $g$ is unique up to isomorphism over $x$.
\end{enumerate}
If $\Phi_X$ is additionally a linear poic-complex, and $\pi\colon\Phi\to\cX$ is a poic-fibration, then we say that $\pi\colon\Phi_X\to\cX$ is a linear poic-fibration. 
\end{defi}

Naturally, we are interested in morphisms of poic-fibrations, for which we make use of the following notation. Let $\pi\colon \Phi\to \cX$ and $\rho\colon \Psi\to \cY$ denote poic-fibrations. A morphism of poic-fibrations $\mathfrak{f}\colon \pi\to\rho$ consists of: 
\begin{itemize}
    \item A morphism of poic-complexes $\mathfrak{f}^\pcmplxs\colon \Phi\to\Psi$ (the superscript stands for complexes).
    \item A morphism of poic-spaces $\mathfrak{f}^\pspcs\colon\cX\to \cY$ (the superscript stands for spaces).
    \item These are subject to the condition $\rho\circ \mathfrak{f}^\pcmplxs = \mathfrak{f}^\pspcs\circ \pi$.
\end{itemize}
In addition, if $\pi\colon \Phi_X\to\cX$ and $\rho\colon\Psi_Y\to\cY$ are linear poic-fibrations, then $\mathfrak{f}^\pcmplxs$ must also be a morphism linear poic-complexes. 

Let $\pi\colon\Phi_X\to\cX$ be a linear poic-fibration. The first and second condition imply that the fibration $\pi$ induces a surjective map of sets $[\pi]\colon [\Phi]\to [\cX]$ that preserves the dimension of the associated cones.
\begin{defi}
Let $k\geq 0$ be an integer. A \emph{$k$-dimensional rational $\pi$-equivariant Minkowski weight} $\omega$ is a weight $\omega\in M_k(\Phi_X)_\bbQ$ such that $\omega$ is constant on the fibers of $[\pi]$. The set of $\pi$-equivariant Minkowski weights is denoted by $M_k(\cX_{\pi,X})_\bbQ$, which is in fact a subgroup of $M_k(\Phi_X)_\bbQ$.
\end{defi}

As in the case of linear poic-complexes subdivisions of $\Phi$ are necessary to introduce $\pi$-equivariant cycles. However, we are not interested in general subdivisions of $\Phi_X$, but rather on subdivisions compatible with the poic-fibration. We call these \emph{$\pi$-compatible subdivisions} and introduce them in Definition 4.3.3 of \cite{DARB1}. The $\pi$-compatible subdivisions of $\Phi$ exhibit the same behaviour with respect to $\pi$-equivariant Minkowski weights as the standard subdivisions of $\Phi$ with respect to Minkowski weights. We consider the category $\pi\textnormal{-}\Subd(\Phi)$ of $\pi$-compatible subdivisions of $\Phi$, then for any integer $k\geq0$ the $k$-dimensional $\pi$-equivariant Minkowski weights give rise to a representable functor $M_k(\cX_{\pi,\bullet})_\bbQ\colon          \pi\textnormal{-}\Subd(\Phi)^{\op}\to \bbQ\text{-Vect}$.
\begin{defi}
    If $\pi:\Phi_X\to\cX$ is a linear poic-fibration, the \emph{group of tropical $k$-cycles of the fibration} is defined as the colimit of $M_k(\cX_{\pi,\bullet})_\bbQ$ over $\pi\textnormal{-}\Subd(\Phi)^{\op}$.
\end{defi}

It is also shown in Lemma 4.3.11 of loc. cit., that if $\pi\colon\Phi_X\to\cX$ is a linear poic-fibration, then the natural inclusions give rise to an injective linear map $\pi_k^*\colon Z_k(\cX_{\pi,X})_\bbQ\to Z_k(\Phi_X)_\bbQ$. We close this subsection with the corresponding definition of irreducible linear poic-fibrations.

\begin{defi}
A linear poic-fibration $\pi\colon \Phi_X\to \cX$ (with $\Phi_X$ pure of dimension $n$) is called \emph{irreducible}, if $Z_n(\cX_{\pi,X})_\bbQ$ is $1$-dimensional.
\end{defi}

\subsection{Spanning tree fibrations and their extensions} We will first describe a poic-space whose underlying cones are given by genus-$g$ $n$-marked graphs together with a linear poic-fibration over it. We then proceed with similar constructions in the case of discrete admissible covers. Throughout this exposition we provide concrete references to \cite{DARB1} and \cite{DARB2} for the interested reader.

Let $n\geq0$ be an integer. A \emph{discrete graph with $n$-marked legs} is a discrete graph $G$ with a bijection $\ell_\bullet(G)\colon \{1,\dots,n\}\to L(G),i\mapsto \ell_i(G)$, which we will call the \emph{marking}. For $1\leq i\leq n$, the \emph{$i$th leg of $G$} is simply the leg $\ell_i(G)\in L(G)$. A \emph{map of $n$-marked graphs} is a map of graphs that commutes with the respective markings.

A map of graphs $f\colon G_1\to G_2$ is a \emph{contraction} if it satisfies the hypotheses of Lemma \ref{lem: map of graphs is contraction}. If both graphs are $n$-marked, then we say that $f$ is a \emph{contraction of $n$-marked graphs} if it is additionally a morphism of $n$-marked graphs. Naturally, composition of contractions is a contraction. In addition, any bijective map of graphs is a contraction.
\begin{defi}
    For an non-negative integers $n$ and $g$ with $2g+n-2>0$, the \emph{category of $n$-marked graphs of genus-$g$} is the category $\bbG_{g,n}$ specified by:
\begin{itemize}
    \item The class of objects consists of the connected genus-$g$ discrete graphs with $n$-marked legs whose vertices are at least $3$-valent.
    \item For two objects $G_1$ and $G_2$ of $\bbG_{g,n}$, the set of morphisms $\Hom_{\bbG_{g,n}}(G_1,G_2)$ consists of the contractions of $n$-marked graphs $f\colon G_1\to G_2$. Composition of morphisms is just composition of maps.
\end{itemize}
    In addition, if $G$ is an object of $\bbG_{g,n}$ its \emph{cone of metrics} $\sigma_G$ is the poic defined as follows. Suppose $K$ is an unmarked subgraph of $G$ whose connected components are non-trivial trees, then $G/K$ is an object of $\bbG_{g,n}$. Furthermore, $E(G/K)=E(G)\backslash K$, and the inclusion $E(G/K)\subset E(G)$ gives rise to a face-embedding
    \begin{equation*}
        \bbR^{E(G/K)}_{\geq 0}\to \bbR^{E(G)}_{\geq0}.
    \end{equation*}
    The poic $\sigma_G$ is defined as the subcone of $\bbR_{\geq0}^{E(G)}$ that contains $\bbR^{E(G)}_{>0}$ and all the images of $\bbR_{>0}^{E(G/K)}$, where $K$ is as above, under the corresponding face-embeddings. 
\end{defi}

This cone of metrics functor is of special importance to us, because it gives rise to a poic-space of interest. Namely, Lemma 2.4.3 of \cite{DARB1} shows that the cone of metrics construction gives rise to a poic-space $\Mtrop_{g,n}\colon \bbG_{g,n}^{\op}\to \textnormal{POIC}$. The case of $g=0$ is relevant as it is a well studied case where $\Mtrop_{0,n}$ carries naturally a linear poic-complex structure, this is explained in full detail in Section 3.3 of \cite{DARB1} (and Section 2.8 of \cite{DARB2}).

We will use this linear poic-complex structure on the case of $g=0$, to define a linear poic-fibration over $\Mtrop_{g,n}$ for arbitrary $g$. We consider the linear poic-complex $\ST_{g,n}$ specified as follows. The underlying poic-space is the product $\Mtrop_{0,n+2g}\times \underline{\mathbb{R}^g_{>0}}$, but the linear poic-complex is obtained from the projection onto $\Mtrop_{0,n+2g}$ and its natural linear poic-complex structure. If $T$ is an object of $\bbG_{0,n+2g}$, then we obtain on object $\st_{g,n}(T)$ of $\bbG_{g,n}$ out of $T$ by joining the $(n+i)$th and the $(n+g+i)$th legs of $T$ into a single edge for $1\leq i\leq g$. In Proposition 4.5.2 of \cite{DARB1} we show that this construction gives rise to an essentially surjective functor
    \begin{equation}
    {\st_{g,n}}\colon  \bbG^{\op}_{0,2g+n}\to \bbG_{g,n}^{\op}.\label{eq: spanning tree functor genus g n marked}
    \end{equation}
For $1\leq i\leq g$ let $e_i$ denote the edge of $\st_{g,n}(T)$ given by glueing the $(n+i)$th and the $(n+g+i)$th legs of $T$. The graph $\st_{g,n}(T)$ has the set of edges $E(\st_{g,A}(T)) = E(T) \cup \{e_1,\dots,e_g\}$, and the poic $\sigma_T$ can therefore be identified with a face of $\sigma_{\st_{g,A}(T)}$. Furthermore, by identifying the $i$th-coordinate with the coordinate corresponding to the edge $e_i$, the poic $\bbR_{>0}^g$ corresponds to the relative interior of the face of $\sigma_{\st_{g,n}(T)}$ given by the $\{e_i\}_{i=1}^g$. In this way, the product $\sigma_T\times \bbR_{>0}^g$ has a natural inclusion morphism into $\sigma_{\st_{g,n} (T)}$. It is shown in Proposition 4.5.6 of loc. cit. that the functor \eqref{eq: spanning tree functor genus g n marked} and these inclusions give rise to a linear poic-fibration:
    \begin{equation}
        \st_{g,n}\colon \ST_{g,n}\to \Mtrop_{g,n}.\label{eq: spanning tree fib}
    \end{equation}
We call this linear poic-fibration the \emph{spanning tree fibration}. It is shown in Proposition 4.5.8 that $\ST_{g,n}$ is pure of rank $3g+n-3$ and the poic-fibration $\st_{g,n}$ is irreducible. 

These poic-spaces come naturally equipped with forgetting-the-marking morphisms of the spanning tree fibrations. These are describe in full detail in Section 4.6 of loc. cit., and the reader is directed there for the corresponding constructions, technicalities and particular properties. For this project, we content ourselves with the heuristics of the only relevant case for our endeavors of these which is forgetting all the markings. Namely, we have a (weakly proper) morphism of linear poic-fibrations (Proposition 4.5.6 of loc. cit.) 
    \begin{equation*}
        \mathfrak{ft}_{\{1,\dots,n\}}\colon \st_{g,n}\to \st_{g},
    \end{equation*}
where the underlying combinatorics are simply given by deleting the marked legs of the graph. To avoid excessive notation, we will refer to the underlying natural transformations of $\mathfrak{ft}_{\{1,\dots,n\}}^{\pspcs}$ and $\mathfrak{ft}_{\{1,\dots,n\}}^{\pcmplxs}$ simply by $\ft_{\{1,\dots,n\},\bullet}$, where $\bullet$ will indicate the $n$-marked graph whose legs are being forgotten.

We now proceed with the case of discrete admissible covers, for which we introduce the following notation throughout.
    \begin{nota}\label{nota: standing hdmmu}
        Let $h,d,m\geq 0$ be integers with $2h+m-2>0$, a vector $\vec{\mu}=(\mu_1,\dots,\mu_m)$ of partitions of $d$, set $n= \sum_{i=1}^m\ell(\mu_i)$, and let $g$ be defined by
    \begin{equation*}
        n+2(g-1) = d\cdot(m+2(h-1)).
    \end{equation*}
We assume that $g$ is an integer, which imposes further conditions on $d$, $m$ and $\vec{\mu}$. It is apparent (but worth mentioning) that this equation for $g$ comes from the Riemann-Hurwitz equality \eqref{eq: Riemann-Hurwitz}.
    \end{nota}
\begin{defi}\label{defi: category achm mu}
    Similar to the categories $\bbG_{g,n}$, we define a category $\bbAC_{d,h}(\vec{\mu})$ consisting of the discrete admissible covers with the previous specifications. Namely, the \emph{category of discrete admissible covers of genus-$h$ $m$-marked curves with ramification $\vec{\mu}$} is the category $\bbAC_{d,h}(\vec{\mu})$ given by the following data:
\begin{itemize}
    \item The objects of $\bbAC_{d,h}(\vec{\mu})$ consist of degree-$d$ discrete admissible covers $\pi\colon G\to H$ where:
    \begin{itemize}
        \item $H$ is an object of $\bbG_{h,m}$.
        \item $G$ is an object of $\bbG_{g,n}$ (it is immediate that $2g+n-2>0$). 
        \item For every $1\leq k\leq m$:
            \begin{equation*}
                \pi\left(\bigcup_{j=1}^{\ell(\mu_k)}\ell_{N+j}(G) \right) =\ell_k(H), 
            \end{equation*}
        where $N=\sum_{i<k}\ell(\mu_i)$.
        \item For every $1\leq k\leq m$ and every $1\leq j\leq \ell(\mu_k)$, the weight of the $\sum_{i<k}\ell(\mu_i)+j$-th leg of $G$ is $\mu_k^j$.
    \end{itemize}
    \item Morphisms are given by pairs of morphisms from the corresponding $\bbG_{g,n}$ and $\bbG_{h,m}$ categories yielding commutative squares of maps of graphs. Composition of morphisms is just composition of tuples of maps.
\end{itemize}
\end{defi}

If $\pi\colon G\to H$ is an object of $\bbAC_{d,h}(\vec{\mu})$, then we set $\src(\pi):=G$ (as in source) and $\trgt(\pi):=H$ (as in target). Given a metric on the target, we determine a metric on the source as specified in the next definition.

\begin{defi}\label{defi: matrix assoc to dac}
    Suppose $\pi\colon G\to H$ is a discrete admissible cover. The \emph{matrix associated to $\pi$} is the linear map $F_\pi:\bbR^{E(H)} \to \bbR^{E(G)},v\mapsto F_\pi(v),$ where for $h\in E(G)$ the $h$-coordinate of $F_\pi(v)$ is defined by
    \begin{equation}
    \left(F_\pi({v})\right)_h = \frac{\mathrm{lcm}\left((d_\pi(e))_{e\in \pi^{-1}(\pi(h))}\right)}{d_\pi(h)}({v})_{\pi(h)},\label{eq: assoc matrix}
    \end{equation}
     where $\mathrm{lcm}(\bullet)$ denotes the least common multiple. The \emph{cone of metrics of $\pi$} is the cone $\sigma_{\pi}$ given as the closure of the graph of the restriction of $F_\pi$ to $\sigma_H^0$ in $\sigma_H\times \sigma_G$.
\end{defi}

The cone of metrics construction gives rise to a poic-space
     \begin{equation*}
        \AC_{d,h,\vec{\mu}}\colon\bbAC_{d,h}^{\op}(\vec{\mu})\to\POIC,\pi\mapsto\sigma_\pi.
     \end{equation*}
Additionally, there are natural morphisms of poic-spaces $\src\colon \AC_{d,h,\vec{\mu}}\to \Mtrop_{g,n}$ and $\trgt\colon \AC_{d,h,\vec{\mu}}\to \Mtrop_{h,m}$ given by preserving the source or the target of the discrete admissible cover respectively.

To construct a poic-fibration over $\AC_{d,h,\vec{\mu}}$, we first observe that the product of linear poic-fibrations is a linear poic-fibration, and therefore the linear poic-fibrations
\begin{align*}
    &\st_{g}\colon \ST_{g}\to \Mtrop_{g},
    &\st_{h,m}\colon \ST_{h,m}\to \Mtrop_{h,m},
\end{align*}
give a linear poic-fibration $\st_{g}\times\st_{h,m}\colon \ST_{g,n}\times\ST_{h,m}\to \Mtrop_{g}\times\Mtrop_{h,m}$. In order to get a poic-fibration over $\AC_{d,h,\vec{\mu}}$ we pullback this poic-fibration through the morphism 
\begin{equation*}
    \left(\mathfrak{ft}_{\{1,\dots,n\}}\circ\mathfrak{src}\right) \times \mathfrak{trgt}\colon \AC_{d,h,\vec{\mu}}\to \Mtrop_g\times\Mtrop_{h,m}.
\end{equation*}

The motivation for this specific pullback is to study tropical  cycles in $\AC_{d,h,\vec{\mu}}$ through a poic-fibration that comes with a source morphism to $\st_{g}$ through which we can pushforward top dimensional cycles (i. e. weakly proper in top dimension). Since the details of this construction are seldom necessary for the purposes of this article, we refer the reader to Section 3.4 of \cite{DARB2} and content ourselves with a quick phrasing of the important definition and result. However, we do honor the notation therefrom. In this case, the previous construction procuces a linear poic-complex $\EST^\varnothing_{d,h,\vec{\mu}}$ together with a morphism of poic-spaces $\est^\varnothing_{d,h,\vec{\mu}}\colon \EST^\varnothing_{d,h,\vec{\mu}}\to \AC_{d,h,\vec{\mu}}$ that is actually a linear poic-fibration (Proposition 3.4.2 of loc. cit.). We refer to this fibration as the \emph{extended spanning tree fibration of unmarked discrete admissible covers of genus-$h$ $m$-marked graphs with ramification $\vec{\mu}$}. From the construction, this fibration comes directly with a (weakly proper in top dimension) morphism of linear poic-fibration $\mathfrak{src}\colon \est^\varnothing_{d,h,\vec{\mu}}\to \st_g$. We now describe a weight of this fibration.

\begin{defi}\label{defi: weight of the cover}
    For a cover $\pi\colon G\to H$ of $\bbAC_{d,h}(\vec{\mu})$ the \emph{standard weight $\varpi^\varnothing_{d,h,\vec{\mu}}(\pi) $ of the cover} is the number given by 
    \begin{equation}
        \varpi^\varnothing_{d,h,\vec{\mu}}(\pi) = \frac{\left(\prod_{e\in E(G)}d_\pi(e)\right)\cdot \left(\prod_{V\in V(G)} H(V)\cdot \mathrm{CF}(V)\right)}{\prod_{h\in E(H)}\mathrm{lcm}\left((d_\pi(e))_{e\in\pi^{-1}(h)}\right)} ,\label{eq: standard weight}
    \end{equation}
    where for $V\in V(G)$:
    \begin{itemize}
        \item $H(V)$ denotes the local rational connected Hurwitz number. Namely, around $V$ the cover induces integer partitions partitions $\lambda_1,\dots,\lambda_k$ of $d_\pi(V)$ and therefore $H(V) = H_{0\to 0}(\lambda_1,\dots,\lambda_k)$\footnote{This Hurwitz number $H_{0\to 0}(\lambda_1,\dots,\lambda_k)$ is just the product of $\frac{1}{d_\pi(V)!}$ with the number of isomorphism classes of branched coverings $f\colon\bbP^1\to\bbP^1$, where the branching locus consists of $m$ distinct points $p_1,\dots,p_m$ and the ramification profile of $f$ at $p_i$ is $\lambda_i$ (for $1\leq i\leq m$).}. 
        \item $\mathrm{CF}(V)$ is the product of factors of $k!$ for each $k$-tuple of adjacent edges or legs of the same weight that map to the same edge of the target.
    \end{itemize}
    Since this is defined over $\bbAC_{d,h}(\vec{\mu})$, we can use the fibration $\est^\varnothing_{d,h,\vec{\mu}}$ and this weight to readily define a $\est^\varnothing_{d,h,\vec{\mu}}$-equivariant weight. We abuse notation and refer to $\varpi^\varnothing_{d,h,\vec{\mu}}$ as an $\est^\varnothing_{d,h,\vec{\mu}}$-equivariant weight.
\end{defi}

The following is an instance of Theorem 4.1.1 of \cite{DARB2}, and we will use this result later in Section 4.

\begin{thm}\label{thm: standard weight}
    The weight $\varpi_{d,h,\vec{\mu}}^\varnothing$ is a top-dimensional $\est^\varnothing_{d,h,\vec{\mu}}$-equivariant Minkowski weight, and its weak pushforward $\mathfrak{src}_*\varpi_{d,h,\vec{\mu}}^\varnothing$ is a $(3h+m-3)$-dimensional $\st_{g}$-equivariant tropical cycle, whenever it is non-trivial. 
\end{thm}

\subsection{Realization} A poic $\sigma$ carries naturally an underlying topological space given by the cone itself. More specifically, for a poic $(\sigma,N)$, we let $|\sigma|$ denote the topological space defined by $\sigma\subset N_\bbR$ with the Euclidean topology. A morphism of poics $f\colon \sigma\to\xi$ induces naturally a continuous map
$|f|\colon |\sigma|\to|\xi|$, and this association gives rise to a functor
\begin{equation}
    |\bullet|\colon \POIC\to \Top,\sigma\xrightarrow{f}\xi\mapsto |\sigma|\xrightarrow{|f|}|\xi|,\label{eqdefi: realization of poics}
\end{equation}
which we call the \emph{realization functor}, and we refer to $|\sigma|$ as the \emph{realization of $\sigma$} (analogously with morphisms). The realization functor \ref{eqdefi: realization of poics} can be extended naturally to poic-spaces. More precisely, if $\cX$ is a poic-space then we have the composition
\begin{equation}
    |\cX|\colon C_X\to\Top,p\mapsto |\cX(p)|.
\end{equation}
Since $C_X$ is essentially finite, the colimit of this functor is representable and a topological space representing this colimit is called \emph{the realization of $\cX$}. 
\section{Tropical covers and divisor theory of metric graphs}

\subsection{Tropical curves and their moduli spaces} Let $G$ be a discrete graph and let $\delta\colon E(G)\to \bbR_{>0}$ be a metric. The \emph{topological realization of $(G,\delta)$} is the metric space $\left|(G,\delta)\right|$ given by identifying an edge $e\in E(G)$ with the interval $[0,\delta(e)]$, and a leg $l\in L(G)$ with the half ray $\bbR_{\geq0}$. For a non-negative integer $n\geq0$ with $2g+n-2\geq 0$, a \emph{tropical $n$-marked curve $\Gamma$ of genus $g$} is a topological space obtained as the realization $\Gamma = \left|(G,\delta)\right|$, where $G$ is an object of $\bbG_{g,n}$ and $\delta \in \sigma_G^0$. If $n=0$, we say that $\Gamma$ is a \emph{metric graph}. We additionally use the following notation, if $G$ is an object of $\bbG_{g,n}$ and $\delta$ is a metric on $G$, then $\underline{\left|(G,\delta)\right|}$ denotes the metric graph given by $\left|(\ft_{\{1,\dots,n\}}(G),\ft_{\{1,\dots,n\}}(\delta))\right|$.

The \emph{moduli space of genus-$g$ $n$-marked tropical curves} $\cM_{g,n}^{\trop}$ is the realization of $\Mtrop_{g,n}$. This is a topological space whose points parameterize all tropical genus-$g$ $n$-marked tropical curves up to isometry. If $g=0$, then it is well known (\cite{GathmannKerberMarkwigTFMSTC}, \cite{SpeyerSturmfels}) that the topological space $\cM_{0,n}^{\trop}$ can be embedded as an $(n-3)$-dimensional simplicial fan through the distance map in a $\left(\binom{n}{2}-n\right)$-dimensional vector space. 

\subsection{Rational functions and divisors}
Throughout this section, let $\Gamma$ denote a metric graph. We revisit the definitions and terminology concerning the divisor theory of metric graphs, and gather several basic results. 

\begin{defi}
    A \emph{rational function} $f$ on a metric graph $\Gamma$ is a continuous piecewise integral affine linear function $f\colon \Gamma\to \bbR$. We denote the set of rational functions on $\Gamma$ by $\Rat(\Gamma)$. If $f\in \Rat(\Gamma)$ and $x\in \Gamma$, then the \emph{order of $f$ at $x$} is the integer $\ord_xf$ defined as:
    \begin{equation*}
        \ord_xf := \textnormal{sum of outgoing slopes of }f\textnormal{ at }x.
    \end{equation*}
\end{defi}

Since $\Gamma$ is compact, the integer $\ord_xf$ vanishes for all but finitely many $x\in \Gamma$ (namely, $\ord_xf=0$ if $f$ is locally linear at $x$). The usual tropical operations (we use max convention) make $\Rat(\Gamma)$ into a semiring without an additive identity element zero. Namely, if $f,g\in \Rat(\Gamma)$ then $f\oplus g :=\max\{f,g\}\in \Rat(\Gamma)$ and $f\odot g:= f+g\in \Rat(\Gamma)$, and these operations satisfy the semiring axioms without an additive identity element (this can of course be circumvented by adjoining the constant function with value $-\infty$).

\begin{defi}
    The \emph{group of divisors} $\Div(\Gamma)$ of $\Gamma$ is the free abelian group on the points of $\Gamma$, that is
    \begin{equation*}
        \Div(\Gamma) : = \bbZ\{\Gamma\}.
    \end{equation*}
    An element $D\in \Div(\Gamma)$ is called a \emph{divisor} of $\Gamma$. Given $x\in \Gamma$, we let $[x]\in\Div(\Gamma)$ denote the divisor determined by the point. In this regard, an arbitrary divisor $D\in\Div(\Gamma)$ is of the form
    \begin{equation*}
        D = \sum_{x\in \Gamma} D_x \cdot [x],
    \end{equation*}
    where the $D_x$ are integers, and all but finitely many are zero. If $D\in \Div(\Gamma)$, then the \emph{degree} of $D$ is the integer 
    \begin{equation*}
        \deg(D) := \sum_{x\in \Gamma} D_x.
    \end{equation*}
    By definition of $\Div(\Gamma)$, this assignment gives rise to a (surjective) linear map
    \begin{equation*}
        \deg : \Div(\Gamma)\to \bbZ.
    \end{equation*}
    A divisor $E\in \Div(\Gamma)$ is called \emph{effective} if $E_x\geq0$ for all $x\in \Gamma$. We let $\Eff(\Gamma)$ denote the submonoid of effective divisors of $\Gamma$.  
\end{defi}

Rational functions give rise to divisors by means of which the notion of linear equivalence is introduced. 

\begin{defi}
    For a rational function $f\in \Rat(\Gamma)$, the \emph{divisor associated to $f$} is defined as
    \begin{equation*}
        \div(f) = \sum_{x\in \Gamma}\ord_xf\cdot [x].
    \end{equation*}
    
    It is readily checked that if $f\in\Rat(\Gamma)$, then $\deg(\div(f))= 0$.
\end{defi}

\begin{remark}
    The map principal divisor map $\div\colon\Rat(\Gamma)\to \Div(\Gamma)$ is linear with respect to tropical multiplication. Namely, if $f,g\in \Rat(\Gamma)$, then $\div(f\odot g) = \div(f)+\div(g)$. 
\end{remark}
   
\begin{defi}
    Two divisors $D,D^\prime\in\Div(\Gamma)$ are called \emph{linearly equivalent} if there exists $f\in\Rat(X)$ such that 
    \begin{equation*}
        D-D^\prime = \div(f).
    \end{equation*}
\end{defi}

\begin{defi}
    Let $D\in\Div(\Gamma)$. The \emph{linear system} associated to $D$ is the set 
    \begin{equation*}
        \left|D\right| := \{E\in \Eff(\Gamma) : E= \div(f)+D \textnormal{ for some }f\in \Rat(\Gamma) \}.
    \end{equation*}
    The \emph{rank of $D$} is defined as:
    \begin{equation*}
        \rk (D) := \max\{k\in \bbZ : \left|D-E\right|\neq \varnothing, \textnormal{ for all }E\in\Eff(\Gamma)\textnormal{ with }\deg(E)=k\}.
    \end{equation*}
\end{defi}
\begin{lem}
    For arbitrary $D,D^\prime\in \Div(\Gamma)$ we have that
    \begin{equation}
        \rk(D)+\rk(D^\prime) \leq \rk(D+D^\prime).  \label{eq: inequality1}
    \end{equation}
    Furthermore, if $D^\prime$ is effective, then
    \begin{equation}
        \rk(D+D^\prime)\leq \rk(D)+\deg(D^\prime). \label{eq: inequality2}
    \end{equation}
\end{lem}
\begin{proof}
    For \eqref{eq: inequality1} we show that if $\rk(D)\geq k$ and $\rk(D^\prime)\geq k^\prime$, then $\rk(D+D^\prime)\geq k+k^\prime$. Suppose $\tilde{E}$ is an effective divisor of degree $k+k^\prime$, and write $\tilde{E} = E+E^\prime$ where both $E$ and $E^\prime$ are effective and $\deg(E)=k$, $\deg(E^\prime) = k^\prime$.\\
    Since $\rk(D)\geq k$, there is $f\in \Rat(\Gamma)$ such that $D-E +\div(f)\geq 0$. Analogously, there is $f^\prime\in \Rat(\Gamma)$ such that $D^\prime-E^\prime+\div(f^\prime)\geq0$. It then follows that 
    \begin{align*}
        (D+D^\prime)-\tilde{E} + \div(f\odot f^\prime) &=(D+D^\prime)-\tilde{E} + \div(f) +\div(f^\prime)\\
        &= D+D^\prime-E-E^\prime + \div(f) +\div(f^\prime)\\
        &=D-E+\div(f) + D^\prime-E^\prime+\div(f^\prime)\geq 0.
    \end{align*}
    In conclusion, $\rk(D+D^\prime)\geq k+k^\prime$, as $\tilde{E}$ was arbitrary. \\
    For \eqref{eq: inequality2}, let $D^\prime$ be effective. We proceed analogously and show that if $\rk(D+D^\prime)\geq k$, then $\rk(D)\geq k-\deg(D^\prime)$. If $k-\deg D^\prime\leq -1$, then clearly $\rk(D)\geq k-\deg D^\prime$. So we can assume without loss of generality that $k-\deg D^\prime\geq 0$. Let $E\in \Eff(\Gamma)$ be arbitrary with $\deg(E) = k-\deg D^\prime$. Then $E^\prime := E+D^\prime$ is effective and $\deg(E^\prime) = k$. Given that $\rk(D+D^\prime)\geq k$, there exists $f\in \Rat(\Gamma)$ such that $D+D^\prime-E^\prime +\div(f)\geq0$. Since $D+D^\prime-E^\prime = D-E$, it follows that $D-E+\div(f)\geq 0$.  Therefore $\rk(D) \geq \rk(D+D^\prime)-\deg(D^\prime)$, because $E$ was arbitrary.
\end{proof}
We close with a definition and a well-known theorem intrinsic to the whole divisor theory of metric graphs.
\begin{defi}
    The \emph{canonical divisor $K_\Gamma$ on $\Gamma$} is 
\begin{equation*}
    K_\Gamma:= \sum_{x\in \Gamma}(\val(x)-2)\cdot[x].
\end{equation*}
\end{defi}

\begin{thm}[Tropical Riemann-Roch theorem \cite{GathmannKerberARRTITG}, \cite{MikhalkinZharkovTCTJATF}]
For an arbitrary divisor $D$ of $\Gamma$ the following equation holds:
\begin{equation*}
    \rk(D) -\rk(K_\Gamma-D) = \deg(D)-(g(\Gamma)-1).
\end{equation*}
\end{thm}

\subsection{Tropical admissible covers} Let $\Gamma$ and $\Delta$ denote two metric graphs. We now recall the definition of tropical admissible cover and explain their relationship with discrete admissible covers.

\begin{nota}
For a point $x\in \Gamma$, we let $T_x\Gamma$ denote the set of outgoing unit tangent vectors at $x$.
\end{nota}
\begin{defi}
    A \emph{tropical admissible cover} $\pi\colon \Gamma\to \Delta$ is a continuous surjective map with finite fibers that is locally affine integral and satisfies the following conditions:
     \begin{itemize}
        \item For every $x\in \Gamma$ the following integer is well defined (does not depend on the choice of $\nu^\prime\in T_{\pi(x)}\Delta$): 
        \begin{equation*}
            d_\pi(x) := \sum_{\nu\in T_x\Gamma \cap \pi^{-1}(\nu^\prime)}\left|\mathrm{slope}_\nu(\pi)\right|,
        \end{equation*}
         where $\nu^\prime\in T_{\pi(x)}\Delta$. Here, $\mathrm{slope}_\nu(\pi)$ denotes the integral slope of the map $\pi$ along $\nu$. We refer to $d_\pi(x)$ as the \emph{local degree of $\pi$ at $x\in\Gamma$}.
        
        \item For every $x\in \Gamma$ we have that (non-negativity of the Riemann-Hurwitz number)
        \begin{equation*}
            r_\pi(x):=\val(x)-2-d_\pi(x)\cdot\left(\val\pi(x)-2\right)\geq 0.
        \end{equation*}
        We refer to $r_\pi(x)$ as the \emph{Riemann-Hurwitz number of $\pi$ at $x\in \Gamma$}.
     \end{itemize}
   Analogous to the case of discrete admissible covers (Lemma 3.1.1 of \cite{DARB2}), the weighted sum of the fibers of points of $\Delta$ with weight given by the local degree is a constant integer $\deg(\pi)$ which we call the \emph{degree of $\pi$}.
\end{defi}
A way of producing tropical admissible covers is by means of discrete admissible covers and a metric on the target graph. For this we are forced to introduce an analog of the edge-length matrix of \cite{VargasDraisma}. This matrix is used to determine a metric on the source of a discrete admissible cover out of a metric on the corresponding target in such a way that the map between the corresponding geometric realizations (after forgetting the marked legs) is a tropical admissible cover. 
\begin{const}\label{const: dac to tac}
    Let $\pi\colon G_1\to G_2$ be a discrete admissible cover, and let $\delta$ denote a metric on $G_2$. The \emph{induced metric} $I_\pi(\delta)$ on $G_1$ is given by
    \begin{equation}
        I_\pi(\delta) (e) := \frac{\delta(\pi(e))}{d_\pi(e)}, \textnormal{ for }e\in E(G_1).\label{eq: edgelength matrix}
    \end{equation}
    If $\Gamma_1:= \underline{\left| (G_1,I_\pi(\delta))\right|}$ and $\Gamma_2:= \underline{\left| (G_2,\delta)\right|}$, then $\pi$ naturally gives rise to a continuous surjective map with finite fibers $\underline{\left|\pi\right|}\colon \Gamma_1\to \Gamma_2$ that is locally affine integral. We claim that $\underline{\left|\pi \right|}$ is a tropical admissible cover. Indeed, the first condition on tropical admissible covers follows directly from its discrete analog, whereas the second condition follows from the corresponding equality and the fact that for $V\in G_1$ and $e\in E(G_1)\cap r_{G_1}^{-1}(V)$: 
    \begin{equation*}
        \#r_{G_1}^{-1}(V)\cap \pi^{-1}(\pi(e)) \leq d_\pi(V). 
    \end{equation*}
\end{const}
\begin{remark}
    Let $\pi$ be as in the previous construction. The matrices $F_\pi$ and $I_\pi$ are naturally different, but nonetheless related. The columns of $I_\pi$ and $F_\pi$ are given by edges of $\trgt(\pi)$. The relationship between the matrices $F_\pi$ and $I_\pi$ is the following: the matrix $F_\pi$ is obtained from $I_\pi$ by multiplying the row indexed by $h\in E(G_2)$ by a factor of $\mathrm{lcm}(d_\pi(e))_{e\in\pi^{-1}(\pi(h))}$. In other words, $F_\pi =  I_\pi\cdot D_\pi$ where $D_\pi\colon \bbR^{E(G_2)}\to \bbR^{E(G_2)}$ is the diagonal matrix where for $h\in E(G_2)$ the corresponding diagonal entry is $\mathrm{lcm}(d_\pi(e))_{e\in\pi^{-1}(\pi(h))}$. Furthermore, the composition $\left(\ft_{\{1,\dots,n\},\src(\pi)}\circ I_\pi \right)$ is the edge-length matrix of \cite{VargasDraisma}.
\end{remark}
\begin{prop}
    Every tropical admissible cover arises from a discrete admissible cover as in Construction \ref{const: dac to tac}.
\end{prop}
\begin{proof}
    Let $\pi\colon \Gamma\to \Delta$ denote a tropical admissible cover, and let $G_\Gamma$ and $G_\Delta$ be combinatorial models for $\Gamma$ and $\Delta$ respectively, such that $\pi^{-1}(V(G_\Delta)) = V(G_\Gamma)$. We also assume that $\pi$ is integral linear at every edge of $E(G_\Gamma)$. By itself, the tropical admissible cover $\pi$ gives rise to a harmonic morphism (we abuse notation and denote it by the same letter) $\pi\colon G_\Gamma\to G_\Delta$, where every RH number is non-negative. The proof proceeds by grafting legs to vertices of $G_\Delta$ and its fibers, to make these RH numbers vanish. It is sufficient to explain the procedure at $W\in V(G_\Delta)$. Let $\pi^{-1}(W)=\{V_1,\dots, V_k\}$, and suppose $r_\pi(V_1) = n$. Then we graft $n$ legs to $W$, say $\ell_1$, to $\ell_n$. Now, to each $V_j$ with $j>1$, we graft $d_\pi(V_j)\cdot n$ legs with weight one, where the first $n$ legs lie above $\ell_1$, the next $n$ legs lie above $\ell_2$, and so on. We observe that this does not modify the RH number at $V_j$. In the case of $V_1$ we graft $n$ legs of weight $2$, one above each $\ell_i$, and $(d_\pi(V_1)-2)\cdot n$ legs of weight $1$, where the first $n$ lie above $\ell_1$, the next $n$ lie above $\ell_2$, and so on. We now observe that the valency of $V_i$ has been increased by $(d_\pi(V_1)-1)\cdot n$ and the valency of $W$ has increased by $n$. Therefore, in the RH number of the modification, the grafting of these legs cancels the previous $n$. We now proceed with $V_2$ and the remaining $V_j$'s in the exact same way.
\end{proof}

\subsection{Pull-back and push-forward of divisors} We collect several results on the pull-back and push-forward of divisors through tropical admissible covers. These results are well-known, and since they are central to our endeavors, we provide proofs and concrete statements.

\begin{const}
    Let $\pi\colon \Gamma\to\Delta$ be a tropical admissible cover. For $x\in \Delta$ we set 
    \begin{equation*}
        \pi^*[x] := \sum_{y\in\pi^{-1}(x)} d_\pi(y)\cdot [y],
    \end{equation*}
    and extend linearly to obtain a linear map $\pi^*\colon \Div(\Delta)\to \Div(\Gamma)$, which we call the \emph{pull-back morphism}. The \emph{ramification divisor of $\pi$} is the divisor on $\Gamma$ defined by $R_\pi := \sum_{x\in\Gamma}r_\pi(x)$.
\end{const}
\begin{lem}
    For a tropical admissible cover $\pi\colon \Gamma\to \Delta$ the ramification divisor is effective and satisfies:
    \begin{equation*}
        K_\Gamma = \pi^*K_\Delta+R_\pi.
    \end{equation*}
\end{lem}
\begin{proof}
    This result follows from a direct computation. Namely, we observe that
    \begin{align*}
        K_\Gamma-\pi^*K_\Delta = \sum_{y\in \Gamma}(\val(y)-2 - d_\pi(y)(\val(\pi(y))-2))\cdot [y] = \sum_{y\in \Gamma }r_\pi(y)\cdot[y],
    \end{align*}
    which is the desired equality.
\end{proof}

\begin{prop}\label{prop: pullback}
    Suppose $\pi\colon\Gamma\to\Delta$ is a tropical admissible cover. Then the pull-back morphism maps effective divisors to effective divisors and preserves linear equivalence. Furthermore, if $D\in \Div(\Delta)$ then $\deg(\pi^*D) = \deg(\pi)\cdot \deg(D)$.
\end{prop}
\begin{proof}
    The claim on effective divisors follows directly from the definition of the pull-back morphism. For the preservation of linear equivalence it is sufficient to show that $\pi^*\div (f) = \div(\pi^*f)$ for an arbitrary $f\in \Rat(\Delta)$, where $\pi^*f\in\Rat(\Gamma)$ is given by $f\circ \pi$. Let $f\in\Rat(\Delta)$ and observe that, by definition, we have
    \begin{align*}
        \div(f) &= \sum_{x\in \Delta} \left(\ord_x f\right)\cdot[x],\\
        \pi^*\div(f)&=\sum_{y\in\Gamma} \left(\ord_{\pi(y)} f\right)\cdot d_\pi(y)\cdot[y],\\
        \div(\pi^*f) &= \sum_{y\in\Gamma} \left(\ord_y(\pi^*f)\right)\cdot[y].
    \end{align*}
    The usual chain rule from calculus shows that $\ord_y(\pi^*f) = \ord_{\pi(y)} f\cdot d_\pi(y)$ for $y\in \Gamma$, and therefore $\pi^*\div(f) = \div(\pi^*f)$. If $D\in\Div(\Delta)$, then 
    \begin{equation*}
        \deg(\pi^*D) = \sum_{x\in \Delta}\sum_{y\in\pi^{-1}(x)}d_\pi(y) D_x=\sum_{x\in \Delta}\left(\sum_{y\in\pi^{-1}(x)}d_\pi(y)\right) D_x= \deg(\pi)\cdot \deg(D).
    \end{equation*}
\end{proof}
\begin{const}
    Let $\pi\colon \Gamma\to\Delta$ be a tropical admissible cover. For $y\in \Gamma$ we set $\pi_*[y] := [\pi(y)]$, and extend linearly to obtain a linear map $\pi_*\colon\Div(\Gamma)\to\Div(\Delta)$, which we call the \emph{pushforward morphism}. Similarly, if $f\in \Rat(\Gamma)$, then we let $\pi_*f\colon\Delta\to \bbR$ denote the function
    \begin{equation*}
        \pi_*f(x):=\sum_{y\in \pi^{-1}(x)} f(y).
    \end{equation*}
    It is readily seen that $\pi_*f\in \Rat(\Delta)$ and, furthermore, $\pi_*\div(f) = \div(\pi_*f)$.
\end{const}
\begin{prop}\label{prop: pushforward}
    The pushforward morphism preserves effective divisors and linear equivalence. If $D\in\Div(\Gamma)$, then $\pi_*D = \deg(D)$ and $\rk(D)\leq \rk(\pi_*D)$.
\end{prop}
\begin{proof}
    The preservation of effective divisors follows from the definition of $\pi_*$, whereas the preservation of linear equivalence follows from the previous observation: if $f\in \Rat(\Gamma)$, then $\pi_*\div(f) = \div(\pi_*f)$. If $D\in \Div(\Gamma)$, then 
    \begin{equation*}
        \deg(\pi_*D) =\sum_{x\in \Delta}\left(\sum_{y\in\pi^{-1}(x)}D_y\right) = \sum_{y\in \Gamma}D_y = \deg(D).
    \end{equation*}
    To show the rank inequality, we show that if $\rk(D)\geq k$, then $\rk(\pi_*D)\geq k$. Let $E$ be an effective divisor on $\Delta$ with $\deg E= k$, and let $E^\prime$ be a effective divisor on $\Gamma$ with $\pi_*E^\prime= E$. In particular, $\deg E^\prime =k$ and, by definition, there exists $f\in \Rat(\Gamma)$ such that $D-E^\prime +\div(f) \geq 0$. Therefore $\pi_*D-E+\div(\pi_*f)$ is effective, and since $E$ was arbitrary, it follows that $\rk (\pi_*D)\geq k$.
\end{proof}

\begin{prop}\label{prop: rank of pullback}
    If $\pi\colon \Gamma\to \Delta$ is a tropical admissible cover, then for any $D\in \Div(\Delta)$ we have that $\rk( \pi^* D) \geq \rank(D)$. In particular, if $\Delta$ is a tree, then for any $x\in \Delta$ we have that $\rk(\pi^*[x])\geq 1$.
\end{prop}
\begin{proof}
    Let $D\in \Div(\Delta)$. We will show that if $\rk(D)\geq k$, then $\rk(\pi^*D)\geq k$. Suppose $\rk(D)\geq k$ and let $E$ be an effective divisor on $\Gamma$ with $\deg(E) = k$. From Proposition \ref{prop: pushforward} it follows that $\pi_*E$ is also effective of degree $k$. Since $\rk(D)\geq k$, there exists $f\in \Rat(\Delta)$ such that $D-\pi_*E + \div(f)$ is effective. It is readily observed that $\pi^*\pi_*E-E$ is effective and therefore 
    \begin{equation*}
        \pi^*D-E +\div(\pi^*f) = \pi^*D-\pi^*\pi_*E+\div(\pi^*f)+\pi^*\pi_*E-E
    \end{equation*}
    is also effective. As $E$ was arbitrary, it follows that $\rk(\pi^*D)\geq k$. If $\Delta$ is a tree and $x\in \Delta$, then $[x]$ is an effective divisor of rank $1$ and therefore $\rk(\pi^* [x])\geq 1$.
\end{proof}
\section{Weierstrass points}
In what follows, $\Gamma$ will denote a metric graph. In this section we recall the definition of Weierstrass points on metric graphs, introduce the notion of geometric Weierstrass points and relate these to the former. Subsequently, we make use of the machinery from \cite{DARB1} and \cite{DARB2} to prove the main theorems from this paper. 
\subsection{Traditional Weierstrass points} Let $x\in \Gamma$ be an arbitrary point. It follows from the Tropical Riemann-Roch theorem that if $k\geq 2g-1$, then $\rk(k\cdot [x]) = k-g$. Furthermore, \eqref{eq: inequality1} and \eqref{eq: inequality2} show that if $n$ a non-negative integer, then 
\begin{equation*}
    \rk(n\cdot [x])\leq \rk((n+1)\cdot[x]) \leq \rk(n\cdot[x])+1.
\end{equation*}
\begin{defi}\label{defi: Wpts}
    A point $x\in \Gamma$ is a \emph{Weierstrass point} if $\rk(g(\Gamma) \cdot [x])\geq 1$.
\end{defi}
\begin{remark}
    If $p\in \Gamma$ is not a Weierstrass point, then $\rk(k\cdot[p]) = k-g$ for $k\geq g$.    
\end{remark}

It is shown in Theorem 4.13 of \cite{BakerSpecialization} that on an arbitrary tropical curve of genus $\geq2$ there exists at least one Weierstrass point. The proof makes use of Lemma 2.8 (Specialization Lemma) therefrom, and the classical fact that Weierstrass points (of algebraic curves) exist on every smooth
curve of genus at least $2$. Separatedly, it is easily shown that there are no Weierstrass points on a metric graph of genus $1$. We now proceed with an example of a Weierstrass and non-Weierstrass points on a genus-$2$ graph.

\begin{example}\label{ex: genus 2}
Consider a metric graph $\Delta$ having a combinatorial model as depicted in Figure \ref{fig: genus 2 metric graph} and where the length of the loop incident to the point $a$ is $4\delta$ for $\delta>0$. In Figure \ref{fig: Wpt of genus 2}, two points of $\Delta$ are depicted, and we claim that: The point $\color{purple}w$ of $\Delta$ is a Weierstrass point, and the point $\color{teal}p$ is a non-Weierstrass point.
\begin{itemize}
    \item To show that $\color{purple}w$ is a Weierstrass points observe that the divisor $2\cdot[w]$ is linearly equivalent to the divisor $[q_1]+[q_2]$ where $q_1$ and $q_2$ are points in this loop with the same distance to $a$. In particular, it is linearly equivalent to $2\cdot[a]$, and therefore, it is linearly equivalent to any effective degree-$2$ divisor supported on the edge given by $a$ and $b$. Since it is linearly equivalent to $2\cdot [b]$, it is also linearly equivalent to the divisor $[q_1^\prime]+[q_2^\prime]$ where $q_1^\prime$ and $q_2^\prime$ are points in the loop incident to $b$ with the same distance to $b$.
    \item To show that $\color{teal}p$ is a non-Weierstrass points, observe that any point in the loop of $b$ has two paths to $b$ and let $v\in\Gamma$ be the unique point where both paths have the same length (put differently, $v$ is the point in the loop maximizing the distance to $b$). We claim that $2\cdot[p]-[v]$ is not linearly equivalent to an effective divisor. Indeed, the divisor $2\cdot [p]-[v]$ is linearly equivalent to $[w]+[p]-[v]$, and this is the unique divisor whose restriction to the loop of $a$ is effective. Therefore, $1$ is the biggest degree at $b$ that we can obtain from a linearly equivalent divisor to $2\cdot [p]-[v]$ whose restriction to the loop of $a$ and the middle edge is effective. But to make $2\cdot[p]-[v]$ effective we must at least obtain a linearly equivalent divisor whose degree at $b$ is $2$. 
\end{itemize}
It is also worth remarking, that the previous arguments also show that any point lying in the edge connecting $a$ and $b$ is a Weierstrass point. In particular, $\Delta$ has infinitely many Weierstrass points.
\begin{figure}[!ht]
    \centering
    \begin{tikzpicture}
        \path (-5,2)--(5,2)--(5,-2)--(-5,-2);
        \draw[line width = 2pt] (-1,0)--(1,0);
        \draw[line width = 2pt] (-2,0) circle (1);
        \draw[line width = 2pt] (2,0) circle (1);
        
        \draw[fill=black] (-1,0) circle (3pt);
        \draw[fill=black] (1,0) circle (3pt);
        \node at (-1.3,0) {$a$};
        \node at (1.3,0) {$b$};
    \end{tikzpicture}
    \caption{The genus $2$ metric graph of Example \ref{ex: genus 2}}
    \label{fig: genus 2 metric graph}
\end{figure}
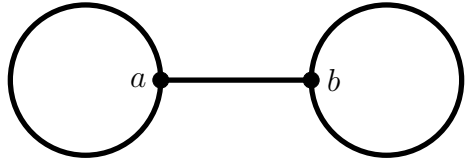

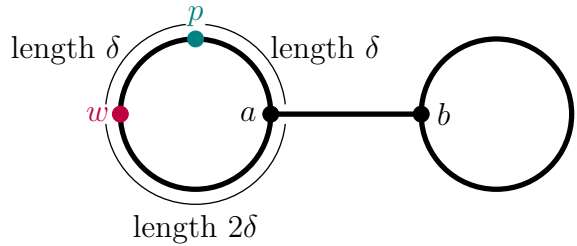
\begin{figure}[!ht]
    \centering
    \begin{tikzpicture}
        \path (-5,2)--(5,2)--(5,-2)--(-5,-2);
        \draw[line width = 0.5pt] (-0.8,0) arc[ start angle = 0, end angle = 90, radius =1.2] node [midway, right] {length $\delta$};
        \draw[line width = 0.5pt] (-2,1.2) arc[ start angle = 90, end angle = 180, radius =1.2] node [midway, left] {length $\delta$};
        \draw[line width = 0.5pt] (-3.2,-0) arc[ start angle = 180, end angle = 360, radius =1.2] node [midway, below] {length $2\delta$};
        \draw[white,fill=white] (-0.8,0) circle (3.5pt);
        \draw[white,fill=white] (-2,1.2) circle (3.5pt);
        \draw[white,fill=white] (-3.2,0) circle (3.5pt);
        \draw[line width = 2pt] (-1,0)--(1,0);
        \draw[line width = 2pt] (-2,0) circle (1);
        \draw[line width = 2pt] (2,0) circle (1);
        
        \draw[fill=black] (-1,0) circle (3pt);
        \draw[fill=black] (1,0) circle (3pt);
        \draw[purple, fill=purple] (-3,0) circle (3pt);
        \node at (-3.3,0) {{\color{purple}$w$}};
        \draw[teal, fill = teal] (-2,1) circle (3pt);
        \node at (-2,1.3) {{\color{teal}$p$}};
        \node at (-1.3,0) {$a$};
        \node at (1.3,0) {$b$};
    \end{tikzpicture}
    \caption{Weierstrass and non-Weierstrass points of the genus $2$ graph from Example \ref{ex: genus 2}}
    \label{fig: Wpt of genus 2}
\end{figure}

\end{example}

We continue with a definition, that will be used at this moment for illustrating the infinitude of Weierstrass points on metric graphs in a family of examples. Separately, this definition will be used in the proof of Theorem \ref{thm: nice computation}.

\begin{defi}
The genus-$g$ discrete graph $O_{g}$ is the object of $\bbG_{g,0}$ given by the following discrete graph:
        \begin{itemize}
        \item It has vertices $V_1,\dots,V_g$ and $W_0,\dots,W_{g-3}$.
        \item There are $g$ loops $l_i$, $1\leq i\leq g$, with $\partial l_i = \{V_i\}$.
        \item There are edges $e_0,\dots, e_{g-4}$ with $\partial e_i = \{W_i,W_{i+1}\}$ for $0\leq i\leq g-4$.
        \item There are edges $h_1,\dots, h_g$, with $\partial h_1=\{W_0,V_1\}$, $\partial h_i =\{V_i,W_{i-2}\}$ for $2\leq i\leq g-1$, and $\partial h_g = \{V_g,W_{g-3}\}$.
    \end{itemize}

In Figure \ref{fig: og} the graph $O_5$ is depicted. In the following example, we will explain how for $g\geq 3$ any metric graph $\Gamma$ having $O_g$ as a model has infinitely many Weierstrass points.  

\begin{figure}[ht!]
        \centering
        \begin{tikzpicture}
            \draw[line width = 2pt] (-6,0) ellipse [x radius = 1 , y radius = 0.5];
            \draw[line width = 2pt] (-5,0)--(-3,0) node [midway, above]{$h_1$};
            \draw[line width = 2pt] (-3,0)--(-1,0) node [midway, above]{$e_0$};
            \draw[line width = 2pt] (-1,0)--(1,0) node [midway, above]{$e_1$};
            \draw[line width = 2pt] (1,0)--(3,0) node [midway, above]{$h_5$};
            \draw[line width =2pt] (-3,0)--(-3,-2) node[midway, left] {$h_2$};
            \draw[line width =2pt] (-1,0)--(-1,-2) node[midway, left] {$h_3$};
            \draw[line width =2pt] (1,0)--(1,-2) node[midway, left] {$h_4$};
            \draw[line width = 2pt] (1,-3) ellipse [y radius = 1 , x radius = 0.5];
            \draw[line width = 2pt] (4,0) ellipse [x radius = 1 , y radius = 0.5];
            \draw[line width = 2pt] (-1,-3) ellipse [y radius = 1 , x radius = 0.5];
            \draw[line width = 2pt] (-3,-3) ellipse [y radius = 1 , x radius = 0.5];
            \draw[fill = black] (-5,0) circle (3pt);
            \draw[fill = black] (-3,0) circle (3pt);
            \draw[fill = black] (-1,0) circle (3pt);
            \draw[fill = black] (-1,-2) circle (3pt);
            \draw[fill = black] (-3,-2) circle (3pt);
            \draw[fill = black] (1,0) circle (3pt);
            \draw[fill = black] (3,0) circle (3pt);
            \draw[fill = black] (1,-2) circle (3pt);
            \node at (-1.5,-2) {$V_2$};
            \node at (-3.5,-2) {$V_3$};
            \node at (0.5,-2) {$V_4$};
            
            \node at (-5,0.5) {$V_1$};
            \node at (-3,0.5) {$W_0$};
            \node at (-1,0.5) {$W_1$};
            \node at (1,0.5) {$W_2$};
            \node at (3,0.5) {$V_5$};
            \node at (-6,0) {$l_1$};
            \node at (-3,-3) {$l_2$};
            \node at (-1,-3) {$l_3$};
            \node at (1,-3) {$l_4$};
            \node at (4,0) {$l_5$};
        \end{tikzpicture}
        \caption{Depiction of $O_{5}$}
        \label{fig: og}
\end{figure}
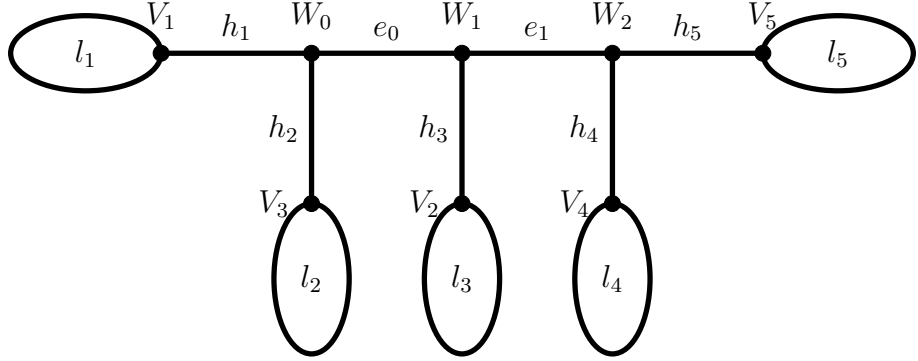
\end{defi}

\begin{example}\label{ex: genus bigger than 3}
    Consider a metric graph $\Gamma$ where the underlying model is $O_g$ for $g\geq 3$. We claim that any point $x\in \Gamma$ is a Weierstrass point. To prove this, we remark the following:
    \begin{itemize}
        \item If $p$ is a point of a genus-$1$ metric graph, then the Tropical Riemann-Roch theorem (or even a direct computation) shows that $\rk(n\cdot [p]) = n-1$.
        \item On a tree, the rank of a divisor coincides with its degree.
    \end{itemize}
    In particular, if $g\geq 3$, then $g\cdot[x]$ is linearly equivalent to $(g-2)\cdot[x]+2\cdot [V_i]$ for any $1\leq i\leq g$. Now, we must show that $g\cdot[x]-[v]$ is linearly equivalent to an effective divisor for arbitrary $[v]\in \Gamma$. Clearly, either $v$ lies on a loop of $\Gamma$, or it lies in the tree defined by the $e_i$'s and the $h_i$'s. In any case, our previous remarks imply that $g\cdot[x]-[v]$ is linearly equivalent to an effective divisor.
\end{example}

\begin{prop}\label{prop: Wpts from tac}
    Suppose $\pi\colon \Gamma\to \Delta$ is a tropical admissible cover where $\deg(\pi) = g(\Gamma)$, and $\Delta$ is a tree. If $y\in \Gamma$ is such that $d_\pi(y)=g(\Gamma)$, then $y$ is a Weierstrass point of $\Gamma$.
\end{prop}
\begin{proof}
    Let $x\in\Delta$ be given by $x=\pi(y)$. Then $g(\Gamma)\cdot[y] = \pi^*[x]$ and since $\Delta$ is a tree, we have that $\rk([x])=1$. It follows from Proposition \ref{prop: rank of pullback} that $\rk\left(g(\Gamma)\cdot [y]\right) \geq \rk([x]) =1$, so $[y]$ is a Weierstrass point. 
\end{proof}
\subsection{Geometric Weierstrass points} We now introduce the central notion of this article, and prove one of the main results (Theorem \ref{thm: nice computation}), from which we later deduce the existence of Weierstrass points on metric graphs of genus $g\geq 2$ (Corollary \ref{cor: existence}).

\begin{defi}
A point $x\in \Gamma$ is called a \emph{geometric Weierstrass point} if there exists:
\begin{itemize}
    \item A discrete admissible cover $\pi\colon G\to H$, with $\deg\pi= g(\Gamma),$ $g(G)= g(\Gamma)$, $H$ is an object of $\bbG_{0,3g}$, and $\pi^{-1}(\ell_1(H)) = \ell_1(G)$ (i. e. its ramification profile is $(g(\Gamma))$).
    \item A metric $\delta$ on $H$ such that $\Gamma = \underline{\left|\left(G, F_\pi(\delta) \right)\right|}$, and $x$ is the point given by $\partial\ell_1(G)$.
\end{itemize} 
\end{defi}

\begin{lem}\label{lem: gWpts are Wpts}
    If $x$ is a geometric Weierstrass point of $\Gamma$, then $x$ is a Weierstrass point.
\end{lem}
\begin{proof}
    Let $x$ be a geometric Weierstrass point of $\Gamma$, let $\pi\colon G\to H$ be the discrete admissible cover from which it arises, and let $\delta$ be the metric on $H$ such that $\Gamma = \underline{\left|\left(G, F_\pi(\delta) \right)\right|}$. To apply Proposition \ref{prop: Wpts from tac} we need to produce a tropical admissible cover $\Gamma\to \Delta$, where $\Delta$ is a tree and such that the local degree at $x$ is $g(\Gamma)$. We recall from Construction \ref{const: dac to tac} that a discrete admissible cover produces a tropical admissible cover where the source graph comes with the induced metric. This means that we have find a metric $\delta^\prime$ on $H$, such that $\Gamma=\underline{\left|(G,I_\pi(\delta^\prime)\right|}$. In this regard, consider the diagonal matrix $D_\pi\colon \bbR^{E(H)}\to \bbR^{E(H)}$ where for $h\in E(H)$ the corresponding diagonal entry is $\mathrm{lcm}(d_\pi(e))_{e\in\pi^{-1}(h)}$. It is readily observed (and has previously been remarked) that $F_\pi= I_\pi\cdot D_\pi$, so that if $\delta^\prime$ is the metric on $H$ such that $D_\pi \cdot \delta= \delta^\prime$, then $\Gamma = \underline{\left| (G, I_\pi (\delta^\prime)) \right|}$. Since these changes on the metric of $H$ do not affect the local degrees, the lemma follows directly from Construction \ref{const: dac to tac} and Proposition \ref{prop: Wpts from tac}.
\end{proof}

\begin{remark}
    The converse does not hold in general. More precisely, it follows from Theorem \ref{thm: nice computation} and Theorem \ref{thm: nice count} that there are only finitely many geometric Weierstrass points on a generic genus-$g$ metric graph. In fact, it is explained in the proof of Theorem \ref{thm: nice computation} that any metric graph whose underlying combinatorial model is $O_g$ has finitely many geometric Weierstrass points (in fact, any genus-$g$ metric graph). However, Examples \ref{ex: genus 2} and \ref{ex: genus bigger than 3} show any metric graph whose underlying combinatorial model is $O_g$ with $g\geq 2$ has infinitely many Weierstrass points.
\end{remark}

Observe that if $\pi\colon G\to H$ is a degree-$g$ discrete admissible cover with $g(H)=0$, $g=g(G)$, and $\#L(H) = 3g$, then the Riemann-Hurwitz equality implies that:
    \begin{itemize}
        \item there is a unique leg of $H$ with ramification profile $(g)$,
        \item all other legs of $H$ have simple ramification $(2,1^{g-2})$.
    \end{itemize}
This previous observation prompts us toward the following notation, in the spirit of \cite{DARB2}:
\begin{nota}\label{nota: settingforgWpts}
    Let $g\geq2$ be given, set $h=0$, $m=3g$, and $\vec{\mu} = (\mu_1,\dots,\mu_m)$, where
    \begin{align*}
        \mu_1 = (g), &&\mu_i = (2,1^{g-2}), \textnormal{ for }1\leq i\leq m.
    \end{align*}
    In this case $n = 1+(3g-1)\cdot (g-1)$, and it is readily verified that:
    \begin{equation}
        n+2(g-1) = g\cdot (3g-2).\label{eq: eq1}
    \end{equation}
\end{nota}

We first require a combinatorial classification to obtain Theorem \ref{thm: nice computation}. To prove the theorem, we proceed as in Theorem 4.3.2 of \cite{DARB2} and show (following our previous notation) that the pushforward of the weight $\varpi_{g,0,\vec{\mu}}^\varnothing$ (Definition 4.1.1 and Theorem 4.1.1 of loc. cit.) through the morphism of poic-fibrations $\mathfrak{src}\colon \est^\varnothing_{g,0,\vec{\mu}}\to \st_g$ (Proposition 3.4.3 of loc. cit.) is a non-trivial multiple of the fundamental weight $[\st_g]$ of $\Mtrop_g$. Following the notation of Theorem 4.1.1 of loc. cit., in our present case of interest, we are not keeping any marked legs, so in the notation therefrom we must take $J=\varnothing$. Furthermore, the degree of the discrete admissible covers is $g$ and coincides with the genus of the source, the genus of the target is zero, and the ramification profiles above the marked legs are as specified in Notation \ref{nota: settingforgWpts}. Before we state, and prove, the theorem, we state a combinatorial classification similar to Lemma 4.2.1 of loc. cit. that is a cornerstone for the proof of Theorem \ref{thm: nice computation}.
\begin{restatable}{lem}{loopandbridge}\label{lem: loopandbridge}
    Following Notation \ref{nota: settingforgWpts}, let $\pi$ be an object of $\bbAC_{g,0}(\vec{\mu})$ of top dimension with $\src(\pi)=O_g$ and such that the composition $\left(\ft_{\{1,\dots,n\},\src(\pi)}\circ F_\pi\right)$ is invertible. If $V\in V(\src(\pi))$ is a vertex giving rise to any one of the $V_i$ vertices ($i=1,\dots,g$), then $\pi(V)$ is incident to a leaf of $\trgt(\pi)$ and $\legval(\pi(V))=1$. Furthermore, there is the following classification: If $e_1$ denotes the leaf of $\trgt(\pi)$ that $\pi(V)$ is incident to, $W$ denotes the other vertex of $\trgt(\pi)$ incident to $e_1$, and $e_2$ denotes the other edge of $\trgt(\pi)$ incident to $\pi(V)$, then one of the following holds (see Figure \ref{fig: classification})
    
    \begin{itemize}
        \item[(A)] The vertex $W$ is incident to the leg with ramification profile $(g)$, the ramification profile of $e_1$ is $(\lambda_1,\lambda_2)$ with $\lambda_1+\lambda_2=g$, and the ramification profile of $e_2$ is $(g)$.
        \item[(B)] The vertex $\pi(V)$ is incident to the leg with ramification profile $(g)$, the ramification profile of $e_1$ is $(1^g)$, the ramification profile above $W$ is $(2,1^{g-2})$, and the ramification profile of $e_2$ is $(g)$.
        \item[(C)] Neither $W$ nor $\pi(V)$ are incident to the leg with ramification profile $(g)$, the ramification profile above $e_1$ is $(1^g)$, and both ramification profiles above $V$ and $e_2$ are $(2,1^{g-2})$. 
    \end{itemize}
    
\end{restatable}

\begin{figure}
        \centering
        \begin{tikzpicture}
            \draw (8.5,6)--(8.5,-6)--(-8.5,-6)--(-8.5,6)--(8.5,6);
            \draw (-8.5,0)--(8.5,0);
            \draw (0,6)--(0,-6);
            
            \draw[line width=2pt] (-5.5,2.5) ellipse [x radius = 1.5, y radius = 1];
            \draw[line width = 2pt] (-4,2.5)--(-1,2.5);
            \draw[dashed, line width = 2pt,red] (-7,2.5) -- (-7.5,3.5) node [at end, above,left] {$g$};
            \draw[dashed, line width = 2pt,blue] (-7,2.5) -- (-7.5,1.5) node [at end, above, left] {$2$};
            \draw[dashed, line width = 2pt,blue] (-7,2.5) -- (-7.5,1);
            \draw[dashed, line width = 2pt,blue] (-7,2.5) -- (-7.5,0.6);
            \draw[dashed, line width = 2pt,blue] (-7,2.5) -- (-7.5,0.2);
            \node at (-5.5,1) {{\color{blue} \parbox{3cm}{$(g-2)$ legs with weight $1$}}};
            \draw[dashed, line width = 2pt,magenta] (-4,2.5) -- (-3.5,2) node [at end, above,right] {$2$};
            \draw[dashed, line width = 2pt,magenta] (-4,2.5) -- (-3.5,1);
            \draw[dashed, line width = 2pt,magenta] (-4,2.5) -- (-3.5,0.6);
            \draw[dashed, line width = 2pt,magenta] (-4,2.5) -- (-3.5,0.2);
            \node at (-2,1) {{\color{magenta} \parbox{3cm}{$(g-2)$ legs with weight $1$}}};
            \node at (-5.5,3.8) {$\lambda_1$};
            \node at (-5.5,1.8) {$\lambda_2$};

            \draw[fill=black] (-4,2.5) circle (3pt);
            \draw[fill=black] (-7,2.5) circle (3pt);
            \node at (-3.8,2.8) {$V$};

            \draw[line width = 2pt] (1+4,2.5) to[out=90, in = 0] (1+2.5,3.8) -- (1+1,3.8);
            \draw[line width = 2pt] (1+4,2.5) to[out=90, in = 0] (1+2.5,4.2) -- (1+1,4.2);
            \draw[line width = 2pt] (1+4,2.5) to[out=90, in = 0] (1+2.5,4.6) -- (1+1,4.6);
            \draw[line width=2pt] (9-5.5,2.5) ellipse [x radius = 1.5, y radius = 1];
            \draw[dashed, line width = 2pt, red] (1+1,3.8)--(1+0.5,4.1);
            \draw[dashed, line width = 2pt, red] (1+1,4.2)--(1+0.5,4.5);
            \draw[dashed, line width = 2pt, red] (1+1,4.6)--(1+0.5,4.9);
            \draw[dashed, line width = 2pt, blue] (1+1,3.8)--(1+0.5,3.6);
            \draw[dashed, line width = 2pt, blue] (1+1,4.2)--(1+0.5,4);
            \draw[dashed, line width = 2pt, blue] (1+1,4.6)--(1+0.5,4.4);
            \draw[fill=black] (2,3.8) circle (3pt);
            \draw[fill=black] (2,4.2) circle (3pt);
            \draw[fill=black] (2,4.6) circle (3pt);

            \node at (6.5,4) {\parbox{3cm}{$(g-2)$ branches}};
            
            \draw[line width = 2pt] (9-4,2.5)--(9-1,2.5);
            \draw[dashed, line width = 2pt,red] (9-7,2.5) -- (9-7.5,3.5) node [at end, above,left] {$2$};
            \draw[dashed, line width = 2pt,blue] (9-7,2.5) -- (9-7.5,1.5) node [at end, above, left] {$2$};
            \draw[dashed, line width = 2pt,magenta] (9-4,2.5) -- (9-3.5,2) node [at end, above,right] {$g$};
            
            \draw[fill=black] (9-4,2.5) circle (3pt);
            \draw[fill=black] (9-7,2.5) circle (3pt);
        
            \node at (5.2,2.8) {$V$};

            \draw[line width=2pt] (-5.5,-4.5) ellipse [x radius = 1.5, y radius = 1];
            \draw[line width = 2pt] (-4,-4.5)--(-1,-4.5) node [midway, above] {$2$};
            \draw[dashed, line width = 2pt,blue] (-7,-4.5) -- (-7.5,-5.5) node [at end, above,left] {$2$};
            \draw[dashed, line width = 2pt,red] (-7,-4.5) -- (-7.5,-3.5) node [at end, above, left] {$2$};
            \draw[dashed, line width = 2pt,magenta] (-4,-4.5) -- (-3.5,-5) node [at end, above,right] {$2$};

            \draw[dashed, line width=2pt,red](-7,-2)--(-7.5,-1.8);
            \draw[dashed, line width=2pt,blue](-7,-2)--(-7.5,-2.2);
            \draw[dashed, line width=2pt,red](-7,-2.5)--(-7.5,-2.3);
            \draw[dashed, line width=2pt,blue](-7,-2.5)--(-7.5,-2.7);
            \draw[dashed, line width=2pt,red](-7,-3)--(-7.5,-2.8);
            \draw[dashed, line width=2pt,blue](-7,-3)--(-7.5,-3.2);

            \draw[dashed, line width =2pt, magenta] (-4,-2)--(-3.5,-2.2);
            \draw[dashed, line width =2pt, magenta] (-4,-2.5)--(-3.5,-2.7);
            \draw[dashed, line width =2pt, magenta] (-4,-3)--(-3.5,-3.2);
            
            \draw[line width = 2pt] (-1,-2)--(-4,-2);
            \draw[line width = 2pt] (-1,-2.5)--(-4,-2.5);
            \draw[line width = 2pt] (-1,-3)--(-4,-3);
            \draw[line width = 2pt] (-7,-2)--(-4,-2);
            \draw[line width = 2pt] (-7,-2.5)--(-4,-2.5);
            \draw[line width = 2pt] (-7,-3)--(-4,-3);
            
            \draw[fill=black] (-4,-4.5) circle (3pt);
            \draw[fill=black] (-4,-2) circle (3pt);
            \draw[fill=black] (-4,-2.5) circle (3pt);
            \draw[fill=black] (-4,-3) circle (3pt);
            \draw[fill=black] (-7,-4.5) circle (3pt);
            \draw[fill=black] (-7,-2) circle (3pt);
            \draw[fill=black] (-7,-2.5) circle (3pt);
            \draw[fill=black] (-7,-3) circle (3pt);

            \node at (-3,-1.5) {$(g-2)$ branches};
            \node at (-3.8,-4.2) {$V$};

            \draw[dashed, line width = 2pt,red] (2,-3.5)--(1.5,-2.5);
            \draw[dashed, line width = 2pt,blue] (2,-3.5)--(1.5,-4.5);
            \draw[dashed, line width = 2pt,magenta] (5,-3.5)--(6,-4.5);
            \draw[line width=2pt] (2,-3.5)--(5,-3.5) node [midway, above] {$e_1$};
            \draw[line width = 2pt] (5,-3.5)--(8,-3.5) node [midway, above] {$e_2$};
            \node at (5,-3) {$\pi(V)$};
            \node at (2,-3) {$W$};

            \draw[fill=black] (5,-3.5) circle (3pt);
            \draw[fill=black] (2,-3.5) circle (3pt);

            \node at (-2.5,4) {$\lambda_1+\lambda_2=g$};
            \node at (-4.5, 5.5) {Case (A)};
            \node at (4.5, 5.5) {Case (B)};
            \node at (-4.5, -0.5) {Case (C)};
            \node at (4.5, -0.5) {$\trgt(\pi)$};
            \draw (-8.5,5)--(8.5,5);
            \draw (-8.5,-1)--(8.5,-1);
        \end{tikzpicture}
        \caption{Cases arising in Lemma \ref{lem: loopandbridge}}
        \label{fig: classification}
    \end{figure}
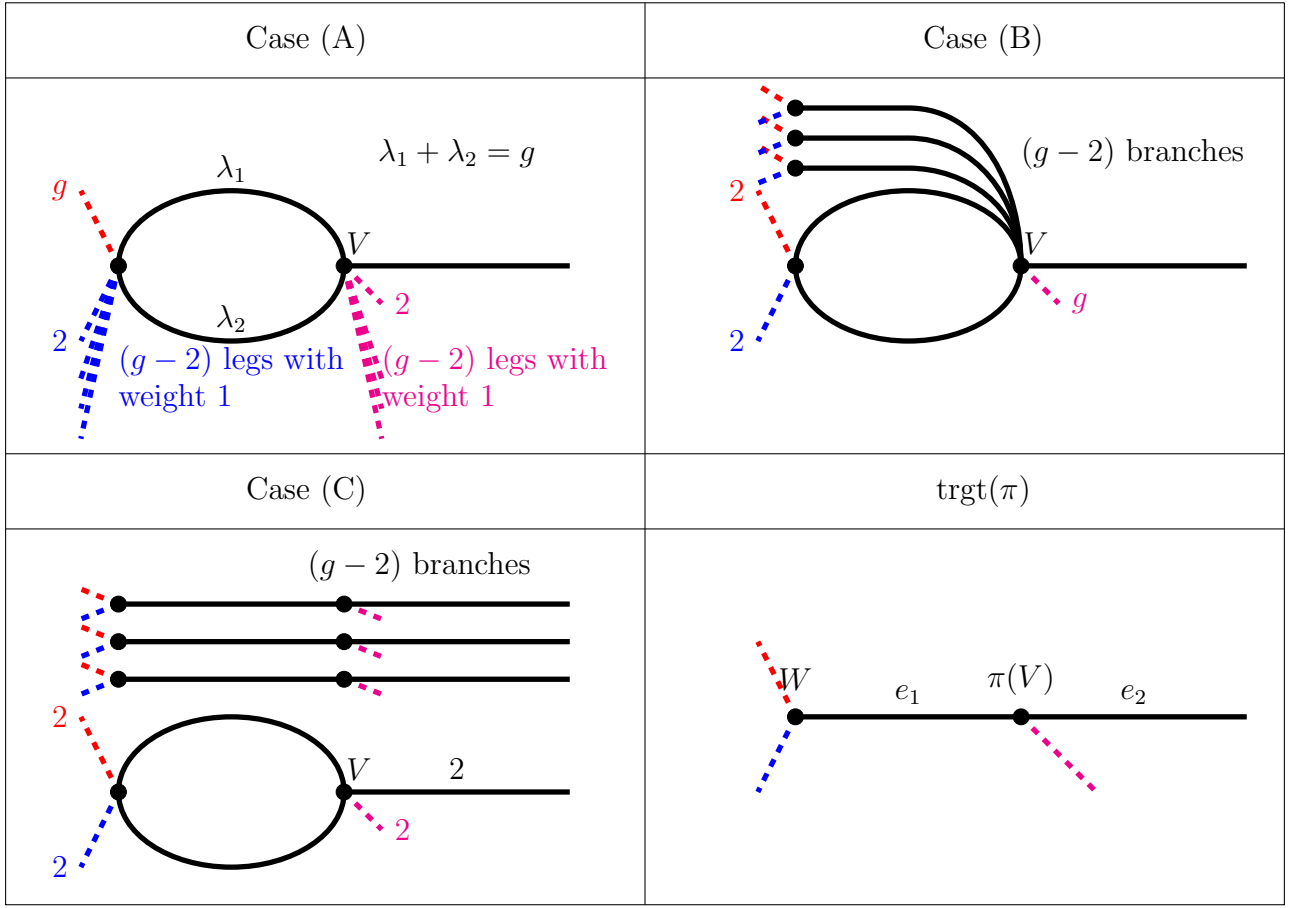

\begin{thm}\label{thm: nice computation}
    Let $g$, $h$, $m$, $\vec{\mu}$, and $n$ be as in Notation \ref{nota: settingforgWpts}. Consider the $\est^\varnothing_{g,0,\vec{\mu}}$-equivariant tropical cycle given by the standard weight $\varpi^\varnothing_{g,0,\vec{\mu}}$ \eqref{eq: standard weight} and the (weakly proper in the top dimension) morphism of linear poic-fibrations $\mathfrak{src}\colon \est^\varnothing_{g,0,\vec{\mu}}\to \st_g$. Then
    \begin{equation*}
        \mathfrak{src}_*\varpi^\varnothing_{g,0,\vec{\mu}} = \left((3g-1)! \cdot ((g-2)!)^{(3g-1)} \cdot (g^3-g)\cdot [\st_{g}]\right). 
    \end{equation*}
\end{thm}

\begin{proof}
We proceed in a spirit similar to that of the proof of Theorem 4.3.2 of \cite{DARB2}. Namely, from Theorem \ref{thm: standard weight} we know that the weight $\varpi^{\varnothing}_{g,0,\vec{\mu}}$ defines an equivariant $\est^\varnothing_{g,0,\vec{\mu}}$ tropical cycle, and its pushforward $\mathfrak{src}_*\varpi^\varnothing_{g,0,\vec{\mu}}$ is a $3(g-1)$-dimensional $\st_g$-equivariant tropical cycle. Since $\st_g$ is irreducible (Proposition 4.5.8 of \cite{DARB1}), our set-up implies that if this pushforward is non-trivial, then it is immediately a rational multiple of the fundamental cycle. To show its non-triviality, we place ourselves at the very special case of $O_{g}$. In this case, we produce a full combinatorial classification following Lemma \ref{lem: loopandbridge} and use it to compute the multiplicity. In principle, we should consider a $\mathfrak{src}^\pcmplxs$-fine subdivision, but it turns out from this combinatorial classification, just as in Theorem 4.3.2 of loc. cit., that in this very special case of $O_g$ we do not need to subdivide. 

Following our set-up, consider a cover $\pi$ of $\bbAC_{g,0}(\vec{\mu})$ with $\ft_{\{1,\dots,n\}}\src(\pi) = O_g$ and such that the matrix $\left(\ft_{\{1,\dots,n\},\src(\pi)}\circ F_\pi\right)$ is invertible. Notation will be slightly abused, so that vertices and edges of $\src(\pi)$ will be denoted in the same way as those of $O_g$. It follows from Lemma \ref{lem: loopandbridge} that there is an $i$ with $1\leq i\leq g$ such that the edge $h_i$ must have weight $g$, and all the other $h_j$ with $j\neq i$ must have weight $2$ (the condition on the matrix being invertible translates to the cover possibly contributing to the pushforward). 
The vanishing of the RH numbers forces the following: 
\begin{itemize}
    \item[$(\star)$] If $e,e^\prime \in E(O_g)$ are incident to $h_i$, then $g+1=d_\pi(e)+d_\pi(e^\prime)$.
    \item[$(\star\star)$] If $e,e^\prime\in E(O_g)$ are incident to $h_j$ with $j\neq i$, then $d_\pi(e)-d_\pi(e^\prime) = \pm1$.
\end{itemize}
We will show, case by case, that if $d_\pi(h_i)=g$ then there is only one possibility of weight assignment for all other edges.
\begin{itemize}
    \item If $i=1,2$ or $i=g-1,g$. Assume for the sake of simplicity that $i=1$ (the other possibilities in this case are identical). Then $(\star)$ implies that $d_\pi(e_0)= g-1$, and we also have that $d_\pi(e_k)\geq d_\pi(e_{k+1})$ for $0\leq k\leq g-1$. Then $(\star\star)$ implies that $d_\pi(e_k) = g-k+1$, and this is the only possibility.
    \item If $3\leq i\leq g-2$. Then $(\star)$ implies that $g+1=d_\pi(e_{i-3})+d_\pi(e_{i-2})$. Furthermore, we have that 
    \begin{align*}
        &d_\pi(e_0)\leq d_\pi(e_1)\leq \dots \leq d_\pi(e_{i-3}),\\
        &d_\pi(e_{i-2})\geq d_{\pi}(e_{i-1})\geq \dots \geq d_\pi(e_g).
    \end{align*}
    The condition $(\star\star)$ shows that $d_\pi(e_{i-3})\leq i$. However, this same condition together with the decreasing chain of inequalities shows that if $d_\pi(e_{i-3})<i$, then $d_\pi(e_g)<1$. In conclusion, $d_\pi(e_{i-3}) = i$, $d_\pi(e_{i-2}) = g+1-i$, and these equalities, together with $(\star\star)$ and the chains of inequalities, fix the weights of all other edges.
\end{itemize}

We additionally remark that, following Lemma \ref{lem: loopandbridge}, we have a full classification of the ramification profile of the edges $\pi(h_j)$ for $1\leq j\leq g$, and moreover, since the matrix $\left(\ft_{\{1,\dots,n\},\src(\pi)}\circ F_\pi\right)$ is invertible, it necessarily follows that the ramification profile of $\pi(e_j)$ is $(d_\pi(e_j),1^{g-d_\pi(e_j)})$ for $1\leq j\leq g$. In addition, all other vertices of $\src(\pi)$ that do not give rise to vertices of $O_g$ (or lie on a loop) must necessarily have local degree $1$ and, in particular, their local Hurwitz numbers and accompanying combinatorial factors are also $1$. We also observe from condition $(\star\star)$ that the local Hurwitz numbers and accompanying combinatorial factors of the $W_j$ vertices with $0\leq j\leq g-3$ are always $1$.

The previous classification implies that $\trgt(\pi)$ is isomorphic up to permutation of the marking to the $3g$-marked tree $T_g$ that we now describe (see Figure \ref{fig: T5} for a depiction of $T_5$):
\begin{itemize}
            \item It has vertices $B_1$, $\dots$ , $B_g$, $B_1^\prime$, $\dots$, $B_g^\prime$, and $M_0$, $\dots$, $M_{g-3}$.
            \item There are $g$ edges $b_i$, $1\leq i\leq g$, with $\partial b_i = \{B_i,B_i^\prime\}$.
            \item There are edges $m_0,\dots, m_{g-4}$ with $\partial m_i = \{M_i,M_{i+1}\}$ for $0\leq i\leq g-4$.
            \item There are edges $q_1,\dots, q_g$, with $\partial q_1=\{M_0,B_1\}$, $\partial q_i =\{B_i,M_{i-2}\}$ for $2\leq i\leq g-1$, and $\partial q_g = \{B_g,M_{g-3}\}$.
            \item There are legs $L_1$, $\dots$, $L_{3g}$, with $\partial L_{3\cdot i-2}=\partial L_{3\cdot i-1} = \{B_i^\prime\}$ and $L_{3\cdot i} = \{B_i\}$ for $1\leq i\leq g$.
\end{itemize}

\begin{figure}[!ht]
    \centering
     \begin{tikzpicture}
            \draw[line width = 2pt] (-7,0) -- (-5,0) node [midway, above]{$b_1$};
            \draw[line width = 2pt] (-5,0)--(-3,0) node [midway, above]{$q_1$};
            \draw[line width = 2pt] (-3,0)--(-1,0) node [midway, above]{$m_0$};
            \draw[line width = 2pt] (-1,0)--(1,0) node [midway, above]{$m_1$};
            \draw[line width = 2pt] (1,0)--(3,0) node [midway, above]{$q_5$};
            \draw[line width =2pt] (-3,0)--(-3,-2) node[midway, left] {$q_2$};
            \draw[line width =2pt] (-1,0)--(-1,-2) node[midway, left] {$q_3$};
            \draw[line width =2pt] (1,0)--(1,-2) node[midway, left] {$q_4$};
            \draw[line width = 2pt] (1,-2)--(1,-4) node [midway, left] {$b_4$};
            
            \draw[line width = 2pt] (3,0)--(5,0)node [midway, above] {$b_5$};
            \draw[line width = 2pt] (-1,-2)--(-1,-4)node [midway, left] {$b_3$};
            \draw[line width = 2pt] (-3,-2)--(-3,-4)node [midway, left] {$b_2$};
            \draw[dashed, line width = 2pt] (-7,0)--(-8,0.5) node [at end, left] {$L_1$};
            \draw[dashed, line width = 2pt] (-7,0)--(-8,-0.5) node [at end,left] {$L_2$};
            \draw[dashed, line width = 2pt] (-5,0)--(-5,-1) node [at end, below] {$L_3$};
            \draw[dashed, line width = 2pt] (-3,-4)--(-4,-5) node [at end, below] {$L_4$};
            \draw[dashed, line width = 2pt] (-3,-4)--(-3,-5) node [at end, below] {$L_5$};
            \draw[dashed, line width = 2pt] (-3,-2)--(-4,-3) node [at end, below] {$L_6$};
            \draw[dashed, line width = 2pt] (-1,-4)--(-2,-5) node [at end, below] {$L_7$};
            \draw[dashed, line width = 2pt] (-1,-4)--(-1,-5) node [at end, below] {$L_8$};
            \draw[dashed, line width = 2pt] (-1,-2)--(-2,-3) node [at end, below] {$L_9$};
            \draw[dashed, line width = 2pt] (1,-4)--(0,-5) node [at end, below] {$L_{10}$};
            \draw[dashed, line width = 2pt] (1,-4)--(1,-5) node [at end, below] {$L_{11}$};
            \draw[dashed, line width = 2pt] (1,-2)--(0,-3) node [at end, below] {$L_{12}$};
            \draw[dashed, line width = 2pt] (5,0)--(6,0.5) node [at end, right]{$L_{13}$};
            \draw[dashed, line width = 2pt] (5,0)--(6,-0.5) node [at end, right]{$L_{14}$};
            \draw[dashed, line width = 2pt] (3,0)--(3,-1) node [at end, right]{$L_{15}$};
            \draw[fill = black] (-7,0) circle (3pt);
            \draw[fill = black] (-5,0) circle (3pt);
            \draw[fill = black] (-3,0) circle (3pt);
            \draw[fill = black] (-1,0) circle (3pt);
            \draw[fill = black] (-1,-2) circle (3pt);
            \draw[fill = black] (-1,-4) circle (3pt);
            \draw[fill = black] (-3,-2) circle (3pt);
            \draw[fill = black] (-3,-4) circle (3pt);
            \draw[fill = black] (1,0) circle (3pt);
            \draw[fill = black] (3,0) circle (3pt);
            \draw[fill = black] (5,0) circle (3pt);
            \draw[fill = black] (1,-2) circle (3pt);
            \draw[fill = black] (1,-4) circle (3pt);
            \node at (-1.6,-2) {$B_3$};
            \node at (-3.6,-2) {$B_2$};
            \node at (0.4,-2) {$B_4$};
            \node at (-1.6,-4) {$B_2^\prime$};
            \node at (-3.6,-4) {$B_3^\prime$};
            \node at (0.4,-4) {$B_4^\prime$};
            
            \node at (-5,0.5) {$B_1$};
            \node at (-7,0.5) {$B_1^\prime$};
            \node at (-3,0.5) {$M_0$};
            \node at (-1,0.5) {$M_1$};
            \node at (1,0.5) {$M_2$};
            \node at (3,0.5) {$B_5$};
            \node at (5,0.5) {$B_5^\prime$};
        \end{tikzpicture}
    \caption{Depiction of $T_5$}
    \label{fig: T5}
\end{figure}
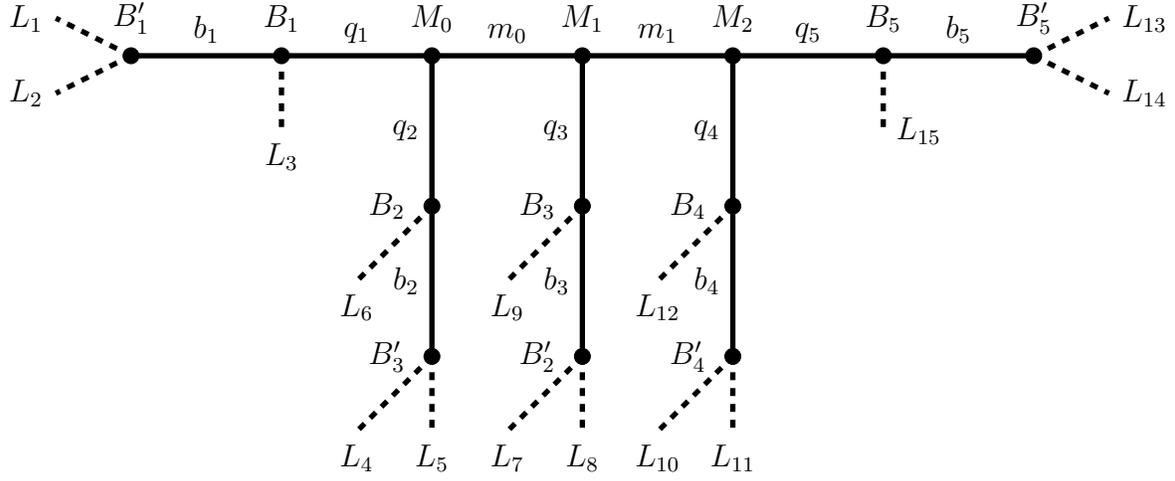

Following our previous classification, we observe that the matrix $\left(\ft_{\{1,\dots,n\},\src(\pi)}\circ F_\pi\right)$ is diagonal, and more interestingly the single entry corresponding to the columns given by the $q_k$ with $1\leq k\leq g$ and the $m_j$ with $0\leq j\leq g-4$ are $1$, because the least common multiple cancels the dividing factor. We have also previously described the behaviour of the local Hurwitz numbers and the accompanying combinatorial factors at all vertices except those lying on the loops. We proceed with through the classification from Lemma \ref{lem: loopandbridge}. 

We observe two cases that we handle independently and depend on the loop of :
    \begin{enumerate}
        \item[(a)] The special loop is as in Case (A) of Lemma \ref{lem: loopandbridge}. We observe the multiplicity $g\cdot 2^{g-1}$ coming from the corresponding loops. Namely, $g-1$ loops behave as in Case (C) of the lemma and contribute each a factor of $2$, and the one loop that behaves as in case (A) contributes the factor of $\lambda_1+\lambda_2=g$. In this case, the remaining Hurwitz numbers are either $1$ or cancel with the corresponding accompanying combinatorial factor.
        \item[(b)] The special loop is as in Case (B) of Lemma \ref{lem: loopandbridge}. In this case we obtain the multiplicity $2^g$ coming from all the possible loops. However, we have a non-trivial contribution from the local Hurwitz numbers and accompanying combinatorial factors. More precisely, if $V$ is the vertex incident to the leg of weight $g$, then $H(V)\cdot \mathrm{CF}(V) = (g-1)!$. The rest of the vertices have local Hurwitz numbers equal to $1$ or these cancel with the corresponding accompanying combinatorial factor.
\end{enumerate}
    To finish we just have to count, correctly, how many different times cases (a) and (b) from above occur and whats their total contribution to the multiplicity. We do this by studying the permutations of the markings. More precisely, permutation of the weight $1$ legs above legs of the target introduces an action of $\left(S_{g-2}\right)^{3g-1}$, and permutation of the markings of the target curve (whose fibers have simple ramification) introduces an action of $S_{3g-1}$.
    \begin{enumerate}
    \item[Case (a):] The remaining accompanying combinatorial factors take into account the action of $\left(S_{g-2}\right)^{3g-1}$ and serve as a weight, so that when we count the corresponding orbits with these factors we obtain a factor of $((g-2)!)^{3g-1}$. Separately, the definition of $T_g$ shows that when we consider the action of $S_{3g-1}$ we will obtain a factor of $\frac{(3g-1)!}{2^{g-1}}$ (the $2^{g-1}$ comes from the $(g-1)$ leaves that have a vertex incident to two legs with simple ramification). There are $g$ loops and hence there are $g$ different possibilities to adjoin the leg of weight $g$. Therefore, this contributes an additional factor of $g$. In addition, permuting the assignment of $\lambda_1$ and $\lambda_2$ to the edges of $\src(\pi)$ yields a different non-isomorphic object of $\bbEAC_{g,0,\vec{\mu}}^{\varnothing}$. Since the number of positive integral ordered tuples adding to $g$ is $(g-1)$, this gives an additional $(g-1)$ factor. In total, all possible instances of case (a) contribute
        \begin{equation*}
            \left((3g-1)!\right)\left(((g-2)!)^{3g-1}\right)\cdot g\cdot (g\cdot(g-1)).
        \end{equation*}

    \item[Case (b):] When we proceed with the analysis of the action of $\left(S_{g-2}\right)^{3g-1}$ we are simply left with:
        \begin{itemize}
            \item a factor of $(g-1)$ stemming out of the local Hurwitz numbers and the accompanying combinatorial factors at the unique vertex with local degree $g$ in the special loop,
            \item a factor of $((g-2)!)^{3g-1}$ coming from the action.
        \end{itemize}
    From the action of $S_{3g-1}$ we obtain a factor of $\frac{(3g-1)!}{2^g}$. As before, there are $g$ loops and hence there are $g$ different possibilities to adjoin the leg of weight $g$, thus yielding an additional factor of $g$. In total, all possible instances of case (b) contribute:
        \begin{equation*} 
            \left((3g-1)!\right)\left(((g-2)!)^{3g-1}\right)\cdot g\cdot (g-1).
        \end{equation*}
    \end{enumerate}
    Now, since 
    \begin{equation*}
        g\cdot(g(g-1))+g\cdot(g-1) = g\cdot(g(g-1)+(g-1)) = (g+1)\cdot g\cdot (g-1),
    \end{equation*}
    it follows that $\mathfrak{src}_*\varpi^{\varnothing}_{g,0,\vec{\mu}} = 
            \left((3g-1)!\right)\left(((g-2)!)^{3g-1}\right)\cdot (g+1)\cdot g\cdot (g-1)$.    
\end{proof}

\begin{remark}
    We observe that in the case of $g=2$, $r=0$, and $J=\varnothing$ Theorem 4.3.2 from \cite{DARB2} is the same as the previous result.
\end{remark}

\begin{cor}\label{cor: existence}
    Every metric graph of genus $g\geq 2$ has a Weierstrass point.
\end{cor}
\begin{proof}
    It follows from Theorem \ref{thm: nice computation} that $\mathfrak{src}_*\varpi^\varnothing_{g,0,\vec{\mu}}$ (Notation \ref{nota: settingforgWpts}) is a non-trivial multiple of the fundamental cycle, hence it is supported on the whole $\Mtrop_{g}$, and we obtain that every genus-$g$ metric graph has at least one geometric Weierstrass point. The corollary now follows from Lemma \ref{lem: gWpts are Wpts}.
\end{proof}

\subsection{Combinatorial Lemmata}
We prove several combinatorial Lemmata on the structure of top-dimensional discrete admissible covers of $\bbAC_{g,0}(\vec{\mu})$ following Notation \ref{nota: settingforgWpts}. Namely, the source of a cover has genus $g\geq 0$, the target of a cover has genus $h=0$, the marking of the target is $m=3g$, and $\vec{\mu}=(\mu_1,\dots,\mu_m)$ with 
\begin{align*}
    \mu_1= (g), &&\mu_i = (2,1^{g-2}), \textnormal{ for }1\leq i\leq m.
\end{align*}
The marking of the target is, therefore, $n=1+(3g-1)\cdot (g-1)$. Assume, furthermore, that $\pi$ is an object of top dimension of $\bbAC_{g,0}(\vec{\mu})$ with $\left(\ft_{\{1,\dots,n\},\src(\pi)}\circ F_\pi\right)$ invertible.

\begin{lem}\label{lem: first}
    If $V$ is a vertex of $\trgt(\pi)$ with $\legval(V)=2$, then the fiber $\pi^{-1}(V)$ is described as one of the following:
    \begin{enumerate}
        \item[(i)] There fiber $\pi^{-1}(V)$ is a single vertex of local degree $g$, and $V$ is incident to $\ell_1(\trgt(\pi))$ (the leg with ramification profile $(g)$). 
        \item[(ii)] There is a unique vertex $\hat{V}$ with $d_\pi(V)= 2$, and all other vertices have local degree $1$. Furthermore, $V$ is incident to legs whose ramification profiles are $(2,1^{g-2})$.
    \end{enumerate}
\end{lem}
\begin{proof}
    Since $\legval(V) = 2$, it follows that $V$ is either incident to $\ell_1(\trgt(\pi))$ or not. In the former, situation we obtain case (i) directly, and in the latter we obtain case (ii) by the same arguments as Lemma 4.3.4 of \cite{DARB2}.
\end{proof}

We continue by recalling, for the sake of convenience, a useful definition from loc. cit.

\begin{defi}
Let $n\geq0$ and $g\geq 0$ be integers with $2g-2\geq 0$ and let $G$ be an object of $\bbG_{g,n}$. An edge $e\in E(G)$ is said to be \emph{expunged after forgetting the marked legs}, if the column corresponding to $e$ in the matrix $\ft_{\{1,\dots,n\}, G}$ is trivial. A vertex $V\in V(G)$ is said to be \emph{expunged after forgetting the marked legs}, if every edge incident to $V$ is expunged after forgetting the marked legs.
\end{defi}

\begin{lem}\label{lem: second}
    Following our notation and assumptions, any vertex of $\src(\pi)$ that is expunged after forgetting the marked legs has local degree $1$. Similarly, every edge of $\src(\pi)$ expunged after forgetting the marked legs has weight $1$.
\end{lem}
\begin{proof}
    The proof is similar to that of Lemma 4.3.5 of loc. cit. up to some modifications stemming out of Lemma \ref{lem: first}. We explain the argument up to the corresponding modifications, and afterward refer directly to Lemma 4.3.5 loc. cit., since it proceeds verbatim. Suppose $V$ is a vertex of $\src(\pi)$ that is expunged after forgetting the marked legs. Observe that $\pi(V)$ is necessarily $3$-valent, since $\pi$ is of top dimension. We have the following classification:
    \begin{itemize}
        \item[$(\bullet)$] If $\legval(\pi(V))=2$, then $d_\pi(V)=1$.
        Indeed, from Lemma \ref{lem: first} we know that $d_\pi(V)=1$, $d_\pi(V)=2$, or $d_\pi(V)=g$. If $d_\pi(V)>1$, then the column of the matrix $\left(\ft_{\{1,\dots,n\},\src(\pi)}\circ F_\pi\right)$ corespodning to the unique edge of $\trgt(\pi)$ incident to $\pi(V)$ would be trivial. Since this matrix is invertible, this is not possible.
        \item[$(\bullet\bullet)$] If $\legval(\pi(V))=1$ and $e_1,e_2\in E(\trgt(\pi))$ denote the two edges incident to $\pi(V)$, then every leg of $\src(\pi)$ incident to $V$ has degree $1$.
        We show that it is not possible for $V$ to be incident to a leg with weight bigger than $1$. 
        \begin{itemize}
            \item If $V$ were incident to the leg with weight $g$, then everything in the fiber of $e_1$ and $e_2$ is expunged. But this is not possible, because $\left(\ft_{\{1,\dots,n\}, \src(\pi)}\circ F_\pi\right)$ is invertible.
            \item If $V$ were incident to a leg with weight $2$, then there must be at least one vertex in the fiber $\pi^{-1}(\pi(V))$ that is not expunged after forgetting the marked legs. If $W$ is one of these vertices, then every leg of $\src(\pi)$ that is incident to $W$ must, necessarily, have weight $1$. However, the vanishing of $r_\pi(W)$ implies that there two edges $e_1^W,e_2^W\in E(\src(\pi))$ incident to $W$ with the same weight and mapping correspondingly to $e_i$. Since this $W\in\pi^{-1}(\pi(V))$ was arbitrary, this shows that the columns of $\left(\ft_{\{1,\dots,n\}, \src(\pi)}\circ F_\pi\right)$ corresponding to the $e_i$ are identical. This is impossible, since this matrix is invertible.
        \end{itemize}  
        \item[$\left(\bullet\bullet\bullet\right)$] If $A,B\in V(\src(\pi))$ are two vertices that are not expunged after forgetting the marked legs, then there cannot be a path between them that consists of edges that are expunged after forgetting the marked legs.
        This follows immediately from the behaviour of the forgetting-the-marking morphisms.
    \end{itemize}
    The rest of the argument follows verbatim that of Lemma 4.3.5 of loc. cit. and is therefore omitted.
\end{proof}

Before finally proving Lemma \ref{lem: loopandbridge}, we have to introduce, for clarity and conciseness, one last definition. 

\begin{defi}\label{defi: add}
We follow our previos notation and let $h\in E(\ft_{\{1,\dots,n\}}(\src(\pi))$. We say that an edge $e\in E(\src(\pi))$ \emph{adds} to the edge $h$ if the $(h,e)$-entry of the matrix $\left(\ft_{\{1,\dots,n\},\src(\pi)}\circ F_\pi\right)$ is non-zero.
\end{defi}

We are finally ready to prove Lemma \ref{lem: loopandbridge}, which is conveniently restated.

\loopandbridge*

\begin{proof}
    First, since $\pi$ is of top dimension, every vertex of $\trgt(\pi)$ is $3$-valent. In the same spirit of Lemma \ref{lem: first} and Lemma \ref{lem: second}, we follow the analogous Lemma 4.3.1 of loc. cit. and, in fact, mimic parts of its proof. We show that $\legval(\pi(V))=2$ and $\legval(\pi(V))=0$ are impossible, and then we prove that $\pi(V)$ is necessarily incident to a leaf. To show the former, we argue by contradiction:
    \begin{itemize}
        \item[$(\circ)$]If $\legval(\pi(V))=2$, then $d_\pi(V)\geq 3$. Therefore, following Lemma \ref{lem: first}, it must be of degree $g$ and $\pi(V)$ is incident to a leg with ramification profile $(g)$. Let $\eta$ denote the ramification profile of the edge incident to $V$, and observe that $\ell(\eta)\geq 3$, because $V$ gives rise to one of the $V_i$ and these are $3$-valent. Since the other leg must have simple ramification, it follows from the vanishing of $r_\pi(V)$ that $\ell(\eta)=2$, but this is a contradiction.
        \item[$(\circ\circ)$] Suppose that $\legval(\pi(V))=0$. Let $a,b,c\in E(\trgt(\pi))$ denote the three edges incident to $\pi(V)$. Let $\pi^{-1}(a)|_V$, $\pi^{-1}(b)|_V$, and $\pi^{-1}(c)|_V$ denote the edges in the corresponding fiber that are incident to $V$. We observe the following:
        \begin{itemize}
            \item If every edge of $\pi^{-1}(a)|_V$ and every edge of $\pi^{-1}(b)|_V$ is expunged after forgetting the marked legs, then $\#\pi^{-1}(c)|_V\geq 3$. But this is impossible due to the vanishing of $r_\pi(V)$ and Lemma \ref{lem: second}.
            \item If every edge of $\pi^{-1}(a)|_V$ is expunged after forgetting the marked legs, then
            \begin{equation*}
                \#\left(\pi^{-1}(b)|_V\cup\pi^{-1}(c)|_V\right)\geq 3.
            \end{equation*}
            The same arguments as above apply and show that this is not possible. 
        \end{itemize}
        From the previous two items, it follows that there are edges $e_a\in\pi^{-1}(a)|_V$, $e_b\in\pi^{-1}(b)|_V$, and $e_c\in\pi^{-1}(c)|_V$, that are not expunged after forgetting the marked legs. Necessarily, two of these edges must add to the (corresponding) loop and the other must add to the edge. Without loss of generality, we assume that $e_a$ and $e_b$ add to the loop while $e_c$ adds to the edge. We now arrive at a contradiction, because each of the edges $e_a$ and $e_b$ traces a path in $\src(\pi)$ that eventually leads to edges $h_a$ and $h_b$ (correspondingly) lying over different leaves of $\trgt(\pi)$ and the columns of the invertible matrix $\left(\ft_{\{1,\dots,n\}, \src(\pi)}\circ F_\pi \right)$ corresponding to the leaves $\pi(h_a)$ and $\pi(h_b)$ would be scalar multiples of each other. In conclusion, $\legval(\pi(V))=0$ is not possible.
    \end{itemize}
    The previous arguments show that $\legval(\pi(V))=1$ is the only possibility. Let $e_1$ and $e_2$ denote the edges of $\trgt(\pi)$ that are incident to $\pi(V)$. Just as before, let $\pi^{-1}(e_1)|_V$ and $\pi^{-1}(e_2)|_V$ denote the set of edges in the corresponding fiber that are incident to $V$. Since $V$ gives rise to a $3$-valent vertex, it follows that there are exactly three edges in $\left(\pi^{-1}(e_1)|_V\cup \pi^{-1}(e_2)|_V\right)$ that are not expunged after forgetting the marked legs. Let us denote these by $a$, $b$, and $c$. The vanishing of $r_h(V)$ implies that these do not lie above the same edge, and, furthermore, two of these must add to the loop and the other to the edge. Without loss of generality, we assume that $a$ and $b$ are the ones that add to the loop, and the same argument as in $(\circ\circ)$ shows that they must lie above the same edge. Hence, we can additionally assume that $a,b\in \pi^{-1}(e_1)|_V$ and $c\in\pi^{-1}(e_2)|_V$. In a similar fashion, the edges traces a path in $\src(\pi)$ that eventually leads to edges $h_a$ and $h_b$ respectively, such that: $h_a$ and $h_b$ add to the loop, and $\pi(h_a)$ and $\pi(h_b)$ are leaves of $\trgt(\pi)$. The same argument, again, as in $(\circ\circ)$ shows that $\pi(h_a)=\pi(h_b)$. Furthermore, since $\left(\ft_{\{1,\dots,n\},\src(\pi)}\circ F_\pi\right)$ is invertible, necessarily $h_a=a$ and $h_b=b$. In particular, $e_1$ has to be a leaf of $\trgt(\pi)$. 
    The classification follows by exhaustion. Namely, cases (A) and (B) arise from assuming that the leg $\ell_1$ of $\trgt(\pi)$ (namely, that with ramification profile $(g)$) is incident to the leaf incident to $\pi(V)$. In this case, the vanishing of $r_h(V)$ makes only these two possiblities feasible. The case (C) is the same as that of Lemma 4.3.1 of loc. cit., and can (separately) also be exhaustively obtained from Lemma \ref{lem: first}, Lemma \ref{lem: second}, and the vanishing of the RH numbers.
\end{proof}
\subsection{Enumeration of geometric Weierstrass points} In this subsection we keep Notation \ref{nota: settingforgWpts} and follow analogous lines to Section $4$ of \cite{DARB2}. That is, we make use of the actions of $(S_{g-2})^{3g-1}$ and the symmetric group $S_{3g-1}$, as in the proof of Theorem \ref{thm: nice computation}, to keep the $(g^3-g)$ factor in Theorem \ref{thm: nice computation}. The actions of these groups introduce an equivalence relation on isomorphism classes of discrete admissible covers of $\bbAC^\varnothing_{g,0}(\vec{\mu})$ (and also on the set of isomorphism classes of the underlying category of $\EST^\varnothing_{g,0,\vec{\mu}})$, and we introduce a multiplicity that is well defined on these. We then show that counting the corresponding equivalence classes with the given multiplicity give the sought enumerative statement (Theorem \ref{thm: nice count}). 

\begin{defi}
    If $\pi$ is a discrete admissible cover of $\bbAC_{g,0}(\vec{\mu})$, then its \emph{multiplicity} is the number
\begin{equation}
    \mt(\pi):=\frac{\varpi^\varnothing_{g,0,\vec{\mu}}\left(\pi\right) \cdot \left|\det(\ft_{\{1,\dots,n\},\src(\pi)}\circ F_\pi)\right|}{\HS(\pi)\cdot \VS(\pi)},\label{eq: multiplicity}
\end{equation}
    where $\VS(\pi)$ and $\HS(\pi)$ are the sizes of the stabilizers of the action of $(S_{g-2})^{3g-1}$ and $S_{3g-1}$ correspondingly.  
\end{defi}

It is a non-trivial fact that this multiplicity is, actually, a non-negative integer. We show this in later in Lemma \ref{lem: mult is int}, and now introduce the previously mentioned equivalence relation to finally state the enumeration result.

\begin{defi}
    The \emph{set of marked discrete admissible covers (of degree-$g$ with ramification $\vec{\mu}$ of trees)} is the set of equivalence classes of $[\bbAC_{g,0}(\vec{\mu})]$ under the following relation: $[\pi]\sim[\pi^\prime]$ if there exist $\vec{\alpha}\in\left(S_{g-2}\right)^{3g-1}$ and $\beta\in S_{3g-1}$ such that $[\pi] = \vec{\alpha}\left(\beta\left([\pi^\prime]\right)\right)$ and $\ft_{\{1,\dots,n\}}\left(\src\left([\pi]\right)\right) =\ft_{\{1,\dots,n\}}\left(\src\left([\pi^\prime]\right)\right)$. Naturally, the multiplicity \eqref{eq: multiplicity} gives a well-defined function on the set of marked discrete admissible covers.
\end{defi}

The relevance of this equivalence relation follows from the observation that a geometric Weierstrass point arises from two discrete admissible covers if and only if these are in the same equivalence class of marked discrete admissible covers.

\begin{thm}\label{thm: nice count}
    A generic genus-$g$ metric graph has $(g^3-g)$-many geometric Weierstrass points counted with multiplicity.
\end{thm}

\begin{proof}
    As has been previously observed, two discrete admissible covers give rise to the same geometric Weierstrass point if and only if they are in the same equivalence class of marked discrete admissible covers. In this regard, we will show that the number of different marked discrete admissible covers is $(g^3-g)$ when counted with their multiplicity \eqref{eq: multiplicity}. We apply Corollary 2.36 of \cite{GathmannKerberMarkwigTFMSTC} (the argument follows verbatim in this set-up) to the realization of $\mathfrak{src}\colon\EST^\varnothing_{g,0}(\vec{\mu})\to\ST_{g,\varnothing}$ following Theorem \ref{thm: nice count}. The fiber of a generic point $p\in\left|\ST_{g,\varnothing}\right|$ has therefore 
    \begin{equation*}
        \left((3g-1)!\cdot ((g-2)!)^{(3g-1)}\cdot (g^3-g)\right)    
    \end{equation*}
    many points counted with multiplicity coming from the weight $\varpi^\varnothing_{g,0,\vec{\mu}}$. By definition, if $q\in  |\mathfrak{src}^\pcmplxs|^{-1}(p)$ and its image in $|\AC_{g,0,\vec{\mu}}|$ lies in the cone of a cover $\pi$ of $\bbAC_{g,0}(\vec{\mu})$, then its multiplicity is
    \begin{equation}
        \varpi^\varnothing_{g,0,\vec{\mu}}(\pi)\cdot \left|\det\left(\ft_{\{1,\dots,n\},\src(\pi)}\circ F_\pi\right)\right|.\label{eq: multiplicity of fibre}
    \end{equation}
    The action of $(S_{g-2})^{3g-1}$ can naturally be extended to $\EST^\varnothing_{g,0}(\vec{\mu})$ and its realization, and it also restricts to an action on the fiber $|\mathfrak{src}^\pcmplxs|^{-1}(p)$. This action preserving the multiplicities of the weight $\varpi^{\varnothing}_{g,0,\vec{\mu}}$, and moves between points in the same class of marked discrete admissible covers. By definition, the size of the stabilizer is $\VS(\pi)$. Similarly, the action of $S_{3g-1}$ gives an action on the fiber $|\mathfrak{src}^\pcmplxs|^{-1}(p)$ that preserves the multiplicities of the weight $\varpi^{\varnothing}_{g,0,\vec{\mu}}$ and moves between points in the same class of marked discrete admissible covers. As before, the size of the stabilizers coincide with $\HS(\pi)$. Since the equivalence relation of marked discrete admissible covers is given by these two actions, it follows that counting each orbit with the multiplicity \eqref{eq: multiplicity} removes both the factor of $((g-2)!)^{(3g-1)}$ coming from the action of $(S_{g-2})^{(3g-1)}$ and the factor of $(3g-1)!$ coming from the action of $S_{3g-1}$, and there are $(g^3-g)$ marked discrete admissible covers with multiplicity.
\end{proof}

We follow with a reprisal of Examples \ref{ex: genus bigger than 3} and \ref{ex: genus 2} but illustrate the corresponding Weirstrass points, and then finish with Lemma \ref{lem: mult is int} and its proof.

\begin{example}\label{ex: genus 3 geometric}
    Consider a metric graph $\Gamma$ where the underlying model is $T_g$. For each $1\leq i\leq g$, let $p_i$ denote the point of the loop $l_i\subset \Gamma$ that maximizes the distance to $V_i$. Then, following the proof of Theorem \ref{thm: nice computation}, one can readily observe that the points $V_i$ and $p_i$ are geometric Weierstrass points. The multiplicity of each $p_i$ is $g\cdot (g-1)$, wherease the multiplicity of each $V_i$ is $(g-1)$. In Figure \ref{fig: genus 3 metric graph gWpts} we have illustrated the genus-$3$ case: the {\color{purple} purple points are the $p_i$ and have multiplicity $6$}, and the {\color{teal} teal points are the $V_i$ and have multiplicity $2$}.
    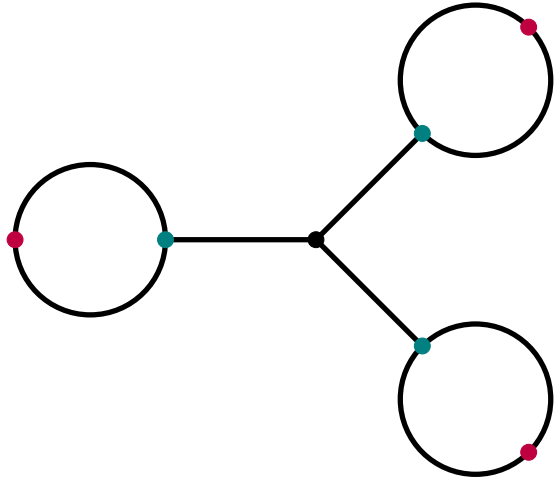
\begin{figure}[!ht]
        \centering
        \begin{tikzpicture}

            \draw[black, line width= 2pt] (-3,0) circle (1);
            \draw[black, line width= 2pt] (2.1213,2.1213) circle (1);
            \draw[black, line width= 2pt] (2.1213,-2.1213) circle (1);
            \draw[line width= 2pt] (-2,0)--(0,0);
            \draw[line width= 2pt] (0,0)--(1.414,1.414);
            \draw[line width= 2pt] (0,0)--(1.414,-1.414);

            \draw[purple,fill=purple] (2.828,2.828) circle (3pt);
            \draw[purple,fill=purple] (2.828,-2.828) circle (3pt);
            \draw[purple,fill=purple] (-4,0) circle (3pt);
            
            \draw[teal,fill=teal] (1.414,1.414) circle (3pt);
            \draw[teal,fill=teal] (1.414,-1.414) circle (3pt);
            \draw[teal,fill=teal] (-2,0) circle (3pt);
            \draw[fill=black] (0,0) circle (3pt);
        \end{tikzpicture}
        \caption{Depiction of geometric Weierstrass points of Example \ref{ex: genus 3 geometric} in genus $3$}
        \label{fig: genus 3 metric graph gWpts}
    \end{figure}

In Figure \ref{fig: genus 2 metric graph gWpts}, the geometric Weierstrass points are depicted: the {\color{purple} purple points are the $p_i$ and have multiplicity $2$}, and the {\color{teal} teal points are the $V_i$ and have multiplicity $1$}.

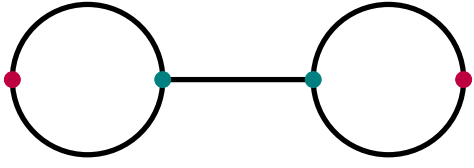
\begin{figure}[!ht]
    \centering
    \begin{tikzpicture}
        \path (-5,2)--(5,2)--(5,-2)--(-5,-2);
        \draw[white,fill=white] (0.8,0) circle (3.5pt);
        \draw[white,fill=white] (3.2,0) circle (3.5pt);
        \draw[white,fill=white] (-0.8,0) circle (3.5pt);
        \draw[white,fill=white] (-3.2,0) circle (3.5pt);
        \draw[line width = 2pt] (-1,0)--(1,0);
        \draw[line width = 2pt] (-2,0) circle (1);
        \draw[line width = 2pt] (2,0) circle (1);
        
        \draw[teal, fill=teal] (-1,0) circle (3pt);
        \draw[teal, fill=teal] (1,0) circle (3pt);
        \draw[purple, fill=purple] (-3,0) circle (3pt);
        \draw[purple, fill=purple] (3,0) circle (3pt);
    \end{tikzpicture}
    \caption{Depiction of geometric Weierstrass points of Example \ref{ex: genus 3 geometric} in genus $2$}
    \label{fig: genus 2 metric graph gWpts}
\end{figure}
\end{example}

\begin{lem}\label{lem: mult is int}
    If $\pi$ is an object of $\bbAC_{g,0}(\vec{\mu})$, then its multiplicity \eqref{eq: multiplicity} is a non-negative integer.
\end{lem}
\begin{proof}
    The proof of this Lemma is very similar to that of Proposition 4.4.2 of \cite{DARB2}, but adapted to our present case. In contrast to this proposition, we do not compare this multiplicity with the Vargas-Draisma multiplicity (in the pertaining case) of \cite{VargasThesis}. Without loss of generality, we can assume that $\left(\ft_{     \{1,\dots,n\},\src(\pi)}\circ F_\pi\right)$ is invertible (otherwise the multiplicity would readily be $0$). Let $A_\pi$ denote the composition $\left(\ft_{\{1,\dots,n\},\src(\pi)}\circ I_\pi\right)$. First, observe that the product of $\left|\det\left(\ft_{\{1,\dots,n\},\src(\pi)}\circ F_\pi\right)\right|$ with the denominator of $\varpi^\varnothing_{g,0,\vec{\mu}}(\pi)$, give precisely $\left|\det\left(A_\phi\right)\right|$. Separetedly, it is clear that $\HS(\pi) =2^N$ and that $\frac{1}{2^N}\det(A_\pi)$ is (because of case (ii) of Lemma \ref{lem: first}) the determinant of the matrix obtained from $A_\pi$ by dividing the columns corresponding to the leaves of $\trgt(\pi)$ where case (ii) of Lemma \ref{lem: first} happens. Therefore, we just have to show that the following expression is an integer
    \begin{equation}
        \frac{\left(\prod_{e\in E(\src(\pi))}d_\pi(e)\right)\cdot\left(\prod_{V\in V(\src(\pi))} H(V)\cdot \mathrm{CF}(V)\right)}{\VS(\pi)}\label{eq: reduction}
    \end{equation}
    that when multiplied with $\frac{1}{2^N}\left|\det(A_\pi)\right|$ gives actually an integer.
    We first seek to understand the local Hurwitz numbers and for this, we follow an exhaustive classification of all the local possibilities arising at the vertices of $\src(\pi)$ supported on Lemma \ref{lem: second}. Consider a vertex $V\in V\left(\src(\pi)\right)$ with $D=d_{\pi}(V)$, we described all the feasible possibilities for the local behaviour of $\pi$ around $V$ in (I), (II), and (III). These cases are accompanied with several depictions, where the coloring describes the edges that lie in the same fiber.
    \begin{enumerate}
        \item[(I)] If $\legval(\pi(V))=0$, then $\pi(V)$ is incident to three legs and, because of Lemma \ref{lem: second}, there is a unique possibility for $\src(\pi)$ around $V$ that we have depicted in Figure \ref{fig: 3valent case}.
        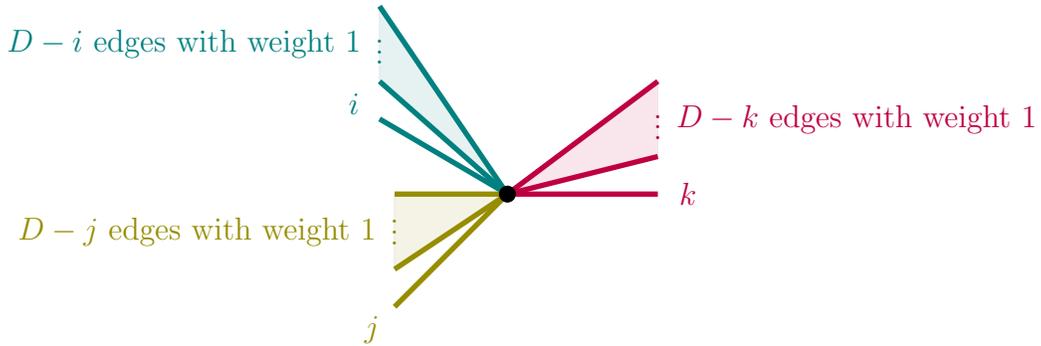
\begin{figure}[ht]
            \centering
            \begin{tikzpicture}
            
                \draw[fill = teal, opacity=0.1,line width=0pt] (0,0)--(-1.7,2.5)--(-1.7,1.5)--(0,0);
                \node[teal] at (-1.7,2) {$\boldsymbol\vdots$};
                \node[teal,anchor=east,align=right] at (-1.8,2) {$D-i$ edges with weight $1$};
                \draw[teal, line width =2pt] (0,0)--(-1.7,2.5);
                \draw[teal, line width =2pt] (0,0)--(-1.7,1.5);
                \draw[teal, line width =2pt] (0,0)--(-1.7,1) node [pos =1.2] {$i$};

                \node[olive] at (-1.5,-0.4) {$\boldsymbol\vdots$};
                \node[olive,anchor=east,align=right] at (-1.6,-0.5) {$D-j$ edges with weight $1$};
                \draw[fill = olive, opacity=0.1,line width=0pt] (0,0)--(-1.5,0)--(-1.5,-1)--(0,0);
                \draw[olive, line width =2pt] (0,0)--(-1.5,0);
                \draw[olive, line width =2pt] (0,0)--(-1.5,-1);
                \draw[olive, line width =2pt] (0,0)--(-1.5,-1.5) node [pos =1.2] {$j$};

                \node[purple] at (2,1) {$\boldsymbol\vdots$};
                \node[purple,anchor=west,align=left] at (2.1,1) {$D-k$ edges with weight $1$};
                \draw[fill = purple, opacity=0.1,line width=0pt] (0,0)--(2,1.5)--(2,0.5)--(0,0);
                \draw[purple, line width =2pt] (0,0)--(2,1.5);
                \draw[purple, line width =2pt] (0,0)--(2,0.5);
                \draw[purple, line width =2pt] (0,0)--(2,0) node [pos =1.2] {$k$};
                \draw[fill = black] (0,0) circle (3pt);
                
            \end{tikzpicture}
            \caption{Unique possiblity of $\src(\pi)$ around $V$ in case (I)}
            \label{fig: 3valent case}
        \end{figure}
        
        Since $r_\pi(V)=0$, necessarily $i+j+k=2D+1$. If one of $i$, $j$, or $k$ is $1$, then the other two must necessarily be equal to $D$ and therefore $H(V) = \frac{1}{D}$. In this case, the accompanying combinatorial factor is killed by $\VS(\pi)$. If $\min\{i,j,k\}>1$, then the local Hurwitz number $H(V)$ is the number of permutations  $\sigma_i$, $\sigma_j$, $\sigma_k$ of $S_D$ with cycle types $(i,1^{D-i})$, $(j,1^{D-j})$, and $(k,1^{D-k})$ respectively, that multiply to the identity and generate a transitive group, divided by $D!$. A direct computation shows that this number is $1$ (else, we refer the reader to case (I) in the proof of Proposition 4.4.2 of \cite{DARB2}). 
        
        \item[(II)] If $\legval(\pi(V))=1$, then we only get the two possibilities depicted in Figures \ref{fig: interesting 2valent case} and \ref{fig: not so interesting 2valent case}. In the second possibility, it is well known that $H(V) = \frac{1}{D}$ and the accompanying combinatorial factor from the {\color{teal} teal edges} is killed by $\VS(\pi)$. In the first possibility, we additionally get the condition that $D=i+j$. In this case, $H(V)$ corresponds to the number of transpositions $\tau$ and permutations $\sigma$ of cycle type $(i)(j)$ of $S_D$, such that $\sigma\tau$ is a $D$-cycle. Again, a direct computation (else the reader is referred to case (II) in the proof of Proposition 4.4.2 of loc. cit.) shows that:
        \begin{itemize}
            \item[(i)] If $i\neq j$, then $H(V)=1$.
            \item[(ii)] If $i=j$, then $H(V)=\frac{1}{2}$.
        \end{itemize}
        In case (ii), the Hurwitz number is cancelled by the accompanying combinatorial factor from the purple edges. In both cases, the accompanying combinatorial factor from the {\color{teal} teal edges} is killed by the $\VS(\pi)$.
        
        \begin{figure}[ht]
        \centering
        \begin{minipage}{0.45\textwidth}
            \centering
            \begin{tikzpicture}
                \node[teal] at (-1.7,2) {$\boldsymbol\vdots$};
                \node[teal,anchor=east,align=right] at (-1.8,2) {$D-2$ edges\\with weight $1$};
                \draw[fill = teal, opacity=0.1,line width=0pt] (0,0)--(-1.7,2.5)--(-1.7,1.5)--(0,0);
                \draw[teal, line width =2pt] (0,0)--(-1.7,2.5);
                \draw[teal, line width =2pt] (0,0)--(-1.7,1.5);
                \draw[teal, line width =2pt] (0,0)--(-1.7,1) node [pos =1.2] {$2$};

                \draw[olive, line width =2pt] (0,0)--(-1.5,-1) node [pos =1.2] {$D$};
                
                \draw[purple, line width =2pt] (0,0)--(2,0.5) node [pos=1.2] {$i$};
                \draw[purple, line width =2pt] (0,0)--(2,-0.5) node [pos =1.2] {$j$};
                \draw[fill = black] (0,0) circle (3pt);
                
            \end{tikzpicture}
            \caption{First possiblity of $\src(\pi)$ around $V$ in case (II)}
            \label{fig: interesting 2valent case} 
            \end{minipage}%
            \begin{minipage}{0.45\textwidth}
            \centering
            \begin{tikzpicture}
                \node[teal] at (-1.7,1.5) {$\boldsymbol\vdots$};
                \node[teal,anchor=east,align=right] at (-1.8,1.5) {$D$ edges\\with weight $1$};\draw[fill = teal, opacity=0.1,line width=0pt] (0,0)--(-1.7,2)--(-1.7,1)--(0,0);
                \draw[teal, line width =2pt] (0,0)--(-1.7,2);
                \draw[teal, line width =2pt] (0,0)--(-1.7,1);

                \draw[fill = black] (0,0) circle (3pt);

                \draw[olive, line width =2pt] (0,0)--(-1.5,-1) node [pos =1.2] {$D$};
                
                \draw[purple, line width =2pt] (0,0)--(2,0) node [pos=1.2] {$D$};
                
            \end{tikzpicture}
            \caption{Second possibility of $\src(\pi)$ around $V$ in case (II)}
            \label{fig: not so interesting 2valent case}
            \end{minipage}
        \end{figure}
        \item[(III)] If $\legval(\pi(V))=2$, then either $\val(V)=3$, $\val(V)=4$, or $\val(V) = g+2$. This is the first case different from the proof of Proposition 4.4.2 of loc. cit.. In this case we get the three possibilities depicted in Figures \ref{fig: first 1valent case} and \ref{fig: second 1valent case}. In the first we get $H(V)=1$, in the second we get $H(V)=\frac{1}{2}$ and this is killed by the accompanying combinatorial factor from the {\color{purple} purple edges}, and in the third we are basically in case (II) from above.
        \begin{figure}[ht]
            \centering
            \begin{minipage}{0.5\textwidth}
            \begin{tikzpicture}
                \draw[teal, line width =2pt] (0,0)--(-1.7,1) node [pos =1.2] {$1$};
                
                \draw[olive, line width =2pt] (0,0)--(-1.5,-1.5) node [pos =1.2] {$1$};
                
                \draw[purple, line width =2pt] (0,0)--(2,0) node [pos =1.2] {$1$};
                \draw[fill = black] (0,0) circle (3pt);
                
            \end{tikzpicture}
            \caption{Case of $\src(\pi)$ around $3$-valent $V$}
            \label{fig: first 1valent case}
            \end{minipage}%
            \begin{minipage}{0.33\textwidth}
            \begin{tikzpicture}
                \draw[teal, line width =2pt] (0,0)--(-1.7,1) node [pos =1.2] {$2$};

                \draw[olive, line width =2pt] (0,0)--(-1.5,-1.5) node [pos =1.2] {$2$};

                \draw[purple, line width =2pt] (0,0)--(2,-0.5) node [pos=1.2]{$1$};
                \draw[purple, line width =2pt] (0,0)--(2,0.5) node [pos =1.2] {$1$};
                \draw[fill = black] (0,0) circle (3pt);
                
            \end{tikzpicture}
            \caption{Possibility of $\src(\pi)$ around $4$-valent $V$}
            \label{fig: second 1valent case}
            \end{minipage}
            \end{figure}

            \begin{figure}[!ht]
                \centering
            \begin{minipage}{0.5\textwidth}
            \begin{tikzpicture}
               \node[teal] at (-1.7,2) {$\boldsymbol\vdots$};
                \node[teal,anchor=east,align=right] at (-1.8,2) {$g-2$ edges\\with weight $1$};
                \draw[fill = teal, opacity=0.1,line width=0pt] (0,0)--(-1.7,2.5)--(-1.7,1.5)--(0,0);
                \draw[teal, line width =2pt] (0,0)--(-1.7,2.5);
                \draw[teal, line width =2pt] (0,0)--(-1.7,1.5);
                \draw[teal, line width =2pt] (0,0)--(-1.7,1) node [pos =1.2] {$2$};
                
                \draw[olive, line width =2pt] (0,0)--(-1.5,-1.5) node [pos =1.2] {$g$};

                \draw[purple, line width =2pt] (0,0)--(2,-0.5) node [pos=1.2]{$\lambda_1$};
                \draw[purple, line width =2pt] (0,0)--(2,0.5) node [pos =1.2] {$\lambda_2$};
                \draw[fill = black] (0,0) circle (3pt);
                \node at (1,-1) {$\lambda_1+\lambda_2=g$};
            \end{tikzpicture}
            \caption{Possibility of $\src(\pi)$ around $(g+2)$-valent $V$}
            \label{fig: third 1valent case}
            \end{minipage}
        \end{figure}
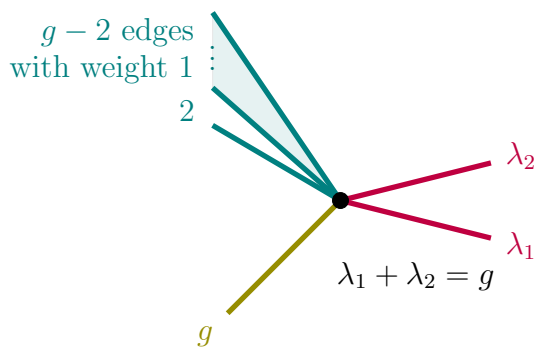
    \end{enumerate}
   From before, the single instance where we have not explained the canceling of the local Hurwitz number is the second possibility of case (II). In this case, the adjacent edges add to a single edge of $\ft_{\{1,\dots,n\}}$. More precisely, if $e\in E(\ft_{\{1,\dots,n\}}\src(\pi))$, then the edges of $\src(\pi)$ that add (Definition \ref{defi: add}) to $e$ form a path of $\src(\pi)$. In this sense, edges of $\ft_{\{1,\dots,n\}}\src(\pi)$ are given by paths of $\src(\pi)$. We further segregate these paths into ones where the weights of the edges are always the same. All instances of the second possibility of case (II) happen as vertices in the interior of a path like these, and we can focus on paths of edges of $\src(\pi)$ where all the edges have weight $D$. In these cases, the factors accompanying combinatorial factors (which are all just $D!$) coming from the {\color{teal} teal edges (or legs)} are killed by the $\VS(\pi)$. All the weights of the edges of $\src(\pi)$ in this path contribute a factor of $D$, and:
   \begin{itemize}
       \item If the path does not go through a leaf, then all but one are killed by the local Hurwitz numbers. 
       \item If the path happens to go through a vertex whose image has leg-valency $2$, then there is no additional contribution (as we have readily explained above) from this vertex and all but two edges are killed by the local Hurwitz number.
   \end{itemize} 
   Here we observe that each of these paths will contribute at least as much as the weight of the edges therein. By definition of the matrix $A_\pi$, these remaining factors when multiplied with $A_\pi$ will actually give an integral matrix and therefore the product we are after will be an integer.\\
   To finish the proof it is only necessary to show that the contribution of the extremal vertices of the paths given is trivial or cancels with what remains of $\VS(\pi)$. Due to Lemma \ref{lem: second}, these extremal vertices are only of the following form:
    \begin{itemize}
        \item The local picture is the unique possibility of case (I), where every edge except the three edges with weights $i$, $j$, and $k$ is expunged after forgetting the marked legs. If $\min\{i,j,k\}>1$, then the local Hurwitz number is $1$ and the accompanying combinatorial factors are cancelled by the remaining factors of $\VS(\pi)$. If (without loss of generality) $i=1$, then the local Hurwitz number is $\frac{1}{D}$ and the accompanying combinatorial factors give $D!$, but in this case $\VS(\pi)$ is precisely $(D-1)!$ coming from the permutations of the {\color{teal}teal edges} that are erased after forgetting the marking. So, the contribution of the extremal vertices having this form is always $1$.
        
        \item The local picture is the first possibility of case (II), where the {\color{teal} teal edges} are actually marked legs, and none of the {\color{purple} purple edges} are erased after forgetting the marking. In both instances, we have readily explained in case (II) what are the contributions and how the accompanying combinatorial factors get killed by $\VS(\pi)$.

        \item The local picture is the third possibility of case (III) with $\lambda_1\neq \lambda_2$, but, just as before, we have readily explained the corresponding behaviour.
        \end{itemize}
        
        Our comprehensive analysis of the local structure explains also how the $\VS(\pi)$ cancels with most of the accompanying combinatorial factors, how the remaining accompanying combinatorial factors cancel with some rational Hurwitz numbers, and how the remaining rational Hurwitz numbers are killed by several weights of the edges. From this exhaustive process, it follows that the multiplicity is necessarily an integer.
\end{proof}


\begin{thebibliography}{GKM09}

\bibitem[Bak08]{BakerSpecialization}
Matthew Baker.
\newblock Specialization of linear systems from curves to graphs (with an appendix by {Brian} {Conrad}).
\newblock {\em Algebra Number Theory}, 2(6):613--653, 2008.

\bibitem[Rob24]{DARB1}
Diego A.~Robayo Bargans.
\newblock A combinatorial extension of tropical cycles, preprint, arXiv: 2410.23474, 2024.

\bibitem[Rob25]{DARB2}
Diego A.~Robayo Bargans.
\newblock Tropical cycles of discrete admissible covers, preprint, arXiv: 2501.16074, 2025.

\bibitem[BN07]{BakerNorine}
Matthew Baker and Serguei Norine.
\newblock Riemann–roch and abel–jacobi theory on a finite graph.
\newblock {\em Advances in Mathematics}, 215(2):766--788, 2007.

\bibitem[DV21]{VargasDraisma}
Jan Draisma and Alejandro Vargas.
\newblock Catalan-many tropical morphisms to trees; {P}art {I}: {C}onstructions.
\newblock {\em J. Symbolic Comput.}, 104:580--629, 2021.

\bibitem[GK08]{GathmannKerberARRTITG}
Andreas Gathmann and Michael Kerber.
\newblock A {R}iemann-{R}och theorem in tropical geometry.
\newblock {\em Math. Z.}, 259(1):217--230, 2008.

\bibitem[GKM09]{GathmannKerberMarkwigTFMSTC}
Andreas Gathmann, Michael Kerber, and Hannah Markwig.
\newblock Tropical fans and the moduli spaces of tropical curves.
\newblock {\em Compos. Math.}, 145(1):173--195, 2009.

\bibitem[MZ08]{MikhalkinZharkovTCTJATF}
Grigory Mikhalkin and Ilia Zharkov.
\newblock Tropical curves, their {J}acobians and theta functions.
\newblock In {\em Curves and abelian varieties}, volume 465 of {\em Contemp. Math.}, pages 203--230. Amer. Math. Soc., Providence, RI, 2008.

\bibitem[SS04]{SpeyerSturmfels}
David Speyer and Bernd Sturmfels.
\newblock The tropical {G}rassmannian.
\newblock {\em Adv. Geom.}, 4(3):389--411, 2004.

\bibitem[VDL22]{VargasThesis}
Alejandro~Jos{\'e} Vargas De~Le{\'o}n.
\newblock {\em Gonality of metric graphs and Catalan-many tropical morphisms to trees}.
\newblock PhD thesis, Universit{\"a}t Bern, Bern, 2022.

\end{thebibliography}
\end{document}